\documentclass[preprint,9pt,sort] {elsarticle}

\usepackage[left=3.5cm,
        right=3.5cm,
        top=3cm,
        bottom=4cm]{geometry}%

\usepackage{amssymb}
\usepackage[x11names]{xcolor}
\usepackage{colortbl}
\usepackage{mathtools}
\usepackage{hyperref}
\usepackage[colorinlistoftodos,bordercolor=black,backgroundcolor= midnightblue!20,linecolor=black,textsize=scriptsize]{todonotes}
\usepackage{algorithm,setspace}
\usepackage{algpseudocode}
\usepackage{amsthm}
\usepackage{cleveref}
\usepackage{autonum}

\journal{Stochastic Processes and their Applications}

\usepackage{tikz}
\usepackage{listings}
\usetikzlibrary{decorations.pathreplacing,calc}
\newcommand{\tikzmark}[1]{\tikz[overlay,remember picture] \node (#1) {};}

\newcommand*{\AddNote}[4]{%
    \begin{tikzpicture}[overlay, remember picture]
        \draw [decoration={brace,amplitude=0.5em},decorate,ultra thick,red]
            ($(#3)!([yshift=1.5ex]#1)!($(#3)-(0,1)$)$) --  
            ($(#3)!(#2)!($(#3)-(0,1)$)$)
                node [align=center, text width=2.5cm, pos=0.5, anchor=west] {#4};
    \end{tikzpicture}
}%

\newcommand{\ceil}[1]{\lceil {#1} \rceil}

\def\balign#1\ealign{\begin{align}#1\end{align}}
\def\baligns#1\ealigns{\begin{align*}#1\end{align*}}
\def\balignat#1\ealign{\begin{alignat}#1\end{alignat}}
\def\balignats#1\ealigns{\begin{alignat*}#1\end{alignat*}}
\def\bitemize#1\eitemize{\begin{itemize}#1\end{itemize}}
\def\benumerate#1\eenumerate{\begin{enumerate}#1\end{enumerate}}
\newenvironment{talign}
 {\align}
 {\endalign}
\newenvironment{talign*}
 {\csname align*\endcsname}
 {\endalign}

\definecolor{darkmidnightblue}{rgb}{0.0, 0.2, 0.4}
\definecolor{darkpowderblue}{rgb}{0.0, 0.2, 0.6}
\definecolor{dukeblue}{rgb}{0.0, 0.0, 0.61}

\hypersetup{
    colorlinks = true,
    citecolor= midnightblue,
    urlcolor= black,
    breaklinks=true,
    linkcolor = dukeblue,
    linkbordercolor = {white},
}

\def\bvartheta{\boldsymbol\vartheta}  
\def\btheta{\boldsymbol\theta}
\def\bxi{\boldsymbol\xi}

\def\bL{\boldsymbol L}
\def\bW{\boldsymbol W}

\def\bG{\boldsymbol G}
\def\wass{{\sf W}}

\def\bV{\boldsymbol V}

\def\bg{\boldsymbol g}
\def\bv{\boldsymbol v}

\newtheorem{theorem}{Theorem}
\newtheorem{assumption}{Assumption}
\newtheorem{proposition}{Proposition}

\newtheorem{lemma}{Lemma}
\newtheorem{corollary}{Corollary}

\def\bB{\boldsymbol{B}}

\def\bfI{\mathbf I}

\def\bfA{\mathbf A}
\def\bfB{\mathbf B}

\def\bfC{\mathbf C}

\def\Ltwo{\mathbb L_2}
\newcommand{\RR}{\mathbb R}

\def\rmd{{\rm d}}
\def\E{\mathbb E}
\newcommand{\grad}{\nabla}
\newcommand{\norm}[1]{\left\| #1 \right\|}

\definecolor{darkmidnightblue}{HTML}{003366}    
\definecolor{midnightblue}{HTML}{0059b3}
\definecolor{chromered}{HTML}{f14233}

\begin{document}

\begin{frontmatter}

\title{Parallelized Midpoint Randomization for Langevin Monte Carlo}

\author[inst1]{Lu Yu}

\affiliation[inst1]{organization={Department of Data Science},%
            addressline={City University of Hong Kong}, 
          country={Hong Kong}
            }

\author[inst2]{Arnak Dalalyan}

\affiliation[inst2]{organization={CREST/ENSAE Paris},%
           country={France}
}

\begin{abstract}
We study the problem of sampling from a target probability density function in frameworks where parallel evaluations of the log-density gradient are feasible. Focusing on smooth and strongly log-concave densities, we revisit the parallelized randomized midpoint method and investigate its properties using recently developed techniques for analyzing its sequential version. Through these techniques, we derive upper bounds on the Wasserstein distance between sampling and target densities. These bounds quantify the substantial runtime improvements achieved through parallel processing.
\end{abstract}


\end{frontmatter}

\section{Introduction}
Parallel computing is widespread in modern machine learning and scientific computing, allowing for the simultaneous execution of multiple tasks, resulting in faster processing times~\cite{grama2003introduction,quinn1994parallel,kumar1994introduction}.
This is particularly crucial for applications involving intensive calculations, such as scientific simulations or big data analysis. Moreover, parallel computing optimizes resource utilization by distributing tasks across multiple processors or computers, reducing the time and energy required for a given task and enhancing efficiency and cost-effectiveness.

Optimization problems, particularly in machine learning and scientific computing, often involve large datasets and complex objective functions. 
Solving such problems on a single processor can be time-consuming. In the field of optimization, leveraging parallel computing to expedite the solution of computationally expensive problems is a common practice. Many optimization algorithms, such as gradient descent, can be parallelized~\cite{recht2011hogwild,jain2018parallelizing}. The workload is distributed among processors, with each processor updating a portion of the model parameters based on its subset of the data.  In deep learning, mini-batch stochastic gradient descent is often used~\cite{abadi2016tensorflow,you2017large}. Parallelization involves distributing these mini-batches across different processors, allowing for efficient model training.

On the other hand, traditional sampling methods often involve sequential processes, which may become computationally burdensome for large datasets or complex models. 
Parallel computing addresses this challenge by distributing the workload across multiple processors, enabling the concurrent execution of sampling tasks and enhancing computational efficiency, thus accelerating the generation of samples in statistical applications. 
Over the past few decades, there has been substantial research on parallel computing for sampling methods, with a notable focus on Monte Carlo methods in Monte Carlo simulations~\cite{rosenthal2000parallel,bhavsar1987design,esselink1995parallel}, particularly in the context of Markov chain Monte Carlo (MCMC)~\cite{neiswanger2014asymptotically,corander2006bayesian,nishihara2014parallel,gonzalez2011parallel}.

However, despite the intrinsic analogy between sampling and optimization~\cite{dalalyan2017further,dalalyan2017theoretical,durmus2019high,durmus2016sampling}, and the extensive studies in parallel schemes for Monte Carlo methods, the application of parallel computing to Langevin Monte Carlo (LMC)~\cite{RobertsTweedie96,Dalalyan14,durmus2017,erdogdu2021convergence,mousavi2023towards,raginsky2017non,Erdogdu18,mou2022improved,erdogdu2022convergence,chewi2023log,chewi2021analysis} remains relatively unexplored\footnote{We note that in a concurrent work, which very recently appeared on arXiv, \cite{anari2024fast} propose parallelizations of Langevin Monte Carlo and kinetic Langevin Monte Carlo for a smooth potential and a target distribution satisfying a
log-Sobolev inequality.}, given that LMC is one of the canonical sampling algorithms in MCMC.
An exception to this is highlighted in the recent work by \cite{shen2019randomized}, where they introduce a parallelized version of the midpoint randomization for kinetic Langevin Monte Carlo~\cite{cheng2018underdamped,dalalyan_riou_2018,shen2019randomized,ma2019there,zhang2023improved,monmarche2021high}. 
This innovative approach notably accelerates the sampling process. 
Building upon the foundations laid by~\cite{shen2019randomized}, we 
explore parallel computing for the midpoint randomization method in Langevin Monte Carlo~\cite{he2020ergodicity,yu2023langevin}. 
Our contributions can be summarized as follows.
\begin{itemize}
    \item  
	\textbf{Parallelized Randomized Midpoint Method 
        for Langevin Monte Carlo~~} 
	We introduce a parallel computing scheme for the randomized midpoint method applied to Langevin Monte Carlo in~\Cref{RLMCp}. We derive the corresponding convergence guarantees in $\wass_2$-distance in~\Cref{thm:rlmc}, providing small constants and explicit dependence on the initialization and choice of the parameters.
    \item  
	\textbf{Parallelized Randomized Midpoint Method for kinetic
        Langevin Monte Carlo~~~} \\
	In~\Cref{thm:rklmc}, we present a comprehensive analysis of the parallel computing for the randomized midpoint method applied to kinetic Langevin Monte Carlo. Compared to previous work, our result offers \textbf{a)} small constants and the explicit dependence on the initialization, \textbf{b)} does not require the initialization to be at minimizer of the potential, \textbf{c)} removes the linear dependence on the number of iterations, \textbf{d)} improves the
    dependence on condition number and the inverse precision level.
\end{itemize}

In summary, this work provides a comprehensive analysis of the parallelized randomized midpoint method as applied to Langevin Monte Carlo. The primary objective is to address the following question in the context of sampling:

\textit{How many units of time are required to ensure that the generated random vector has a distribution within a prescribed distance from the target distribution?}

We investigate this question under two distinct settings:

\begin{description}
    \item[Unlimited parallel units:] Within one unit of time, we assume an unlimited capacity for parallel evaluations of the gradient of the log-density.
    \item[Limited parallel units:] Within one unit of time, we assume a maximum of \(R\) parallel evaluations of the gradient of the log-density, where \(R \in \mathbb{N}\) is a fixed parameter.
\end{description}

\textbf{Notation.}
Let $(\RR^p,\|\cdot\|_2)$ denote the $p$-dimensional 
Euclidean space, and let $\mathcal{B}(\RR^p)$ be its 
Borel $\sigma$-algebra. The letter $\btheta$ represents 
a deterministic vector, while its calligraphic counterpart $\bvartheta$ represents a random vector. Let $\bfI_p$ and $\mathbf 0_p$ denote the $p \times p$ identity and zero matrices, respectively. For two symmetric $p \times p$ matrices $\bfA$ and $\bfB$, we define the relations $\bfA \preccurlyeq\bfB$ and $\bfB\succcurlyeq \bfA$ to mean that $\bfB - \bfA$ is positive semi-definite. 
The gradient and the Hessian of a function $f:\RR^p \rightarrow \RR$ are denoted by $\nabla f$  and 
 $\nabla^2 f$, respectively.
Given any pair of measures $\mu$ and $\nu$ defined on $(\RR^p,\mathcal{B}(\RR^p))$,
the Wasserstein-2 distance between $\mu$ and $\nu$ is defined as
\begin{equation}
\wass_2(\mu,\nu) = \Big(\inf_{\varrho\in \Gamma(\mu,\nu)} \int_{\RR^p\times \RR^p}
\|\btheta -\btheta'\|_2^2 \,\rmd\varrho(\btheta,\btheta')\Big)^{1/2},
\end{equation}
where the infimum is taken over all joint distributions $\varrho$ that have $\mu$ and $\nu$ as marginals. 
The Kullback-Leibler (KL) divergence of $\mu$ from $\nu$ 
is defined as
\begin{align}
D_{\sf KL}(\mu||\nu)
=\begin{cases}
\E_{\bvartheta\sim\mu}\Big[\log\frac{d\mu}{d\nu}(\bvartheta )
\Big] &, \text{ if }\mu \text{ is absolutely continuous with respect to }\nu\, \\
\infty &,\text{ otherwise}\,.
\end{cases}
\end{align}
The ceiling function maps $x\in\RR$ to the smallest 
integer $k$ such that $k\geqslant x$, denoted by 
$\ceil{x}$.

\section{Parallelized randomized midpoint 
discretization: the vanilla Langevin diffusion}
\label{sec:rlmc}

Let the function $f:\RR^p\to\RR$, referred to as the 
potential, be such that $\int_{\mathbb R^p} e^{-f
(\btheta)}\,d\btheta$ is finite. We call target 
distribution the probability distribution having 
the probability density function 
\begin{align}
    \pi(\btheta)\propto \exp\{-f(\btheta)\},\qquad 
    \btheta\in\mathbb R^p.
\end{align}
The goal of sampling is to devise an algorithm that 
generates a random vector in $\mathbb R^p$ from a 
distribution which is close to the target one. 
Throughout the paper, we assume that the 
potential function $f$ is $M$-smooth 
and $m$-strongly convex for some constants $m,M
\in(0,\infty)$ such that $m\leqslant M$.

\begin{assumption}\label{asm:A-scgl}
   The function $f: \mathbb R^p\to\mathbb R$ 
   is twice differentiable, and its Hessian 
   matrix $\nabla^2 f$ satisfies 
\begin{align}
    m\bfI_p\preccurlyeq \nabla^2 f(\btheta)
    \preccurlyeq M\bfI_p,\qquad \forall 
    \btheta\in\mathbb R^p
\end{align}
for some constants $0<m\leqslant M<\infty$. We denote by
$\kappa$ the condition number $M/m$. 
\end{assumption}

Let $\bvartheta_0$ be a random vector drawn from a 
distribution $\nu$ on $\mathbb R^p$ and let $\bW 
=(\bW_t: t\geqslant 0)$ be a $p$-dimensional 
Brownian motion independent of $\bvartheta_0$. 
Using the potential $f$, the random variable 
$\bvartheta_0$ and the process $\bW$, one can define 
the stochastic differential equation
\begin{align}\label{eq:LD}
    d\bL_t^{\sf LD} = -\nabla f(\bL_t^{\sf LD})\,dt 
    + \sqrt{2}\,\rmd\bW_t,
    \quad t\geqslant 0,\qquad \bL^{\sf LD}_0=\bvartheta_0.
\end{align}
This equation has a unique strong solution, which 
is a continuous-time Markov process, termed Langevin 
diffusion. Under some further assumptions on $f$, 
such as strong convexity or dissipativity, the 
Langevin diffusion is ergodic, geometrically
mixing and has $\pi$ as its unique invariant 
density \citep{bhattacharya1978}. In particular, 
the distribution of $\bL_t^{\sf LD}$ converges to 
$\pi$ when $t$ tends to infinity. 

Therefore, we can sample from the distribution defined 
by $\pi$ by using a suitable discretization of the Langevin 
diffusion. The Langevin Monte Carlo (LMC) method is 
based on this idea, combining the aforementioned 
considerations with the Euler discretization. 
Specifically, for small values of $h \geqslant 0$ and 
$\Delta_h\bW_t = \bW_{t+h} - \bW_t$, the following 
approximation holds
\begin{align}
    \bL^{\sf LD}_{t+h} & = \bL^{\sf LD}_t- \int_0^h  
    \nabla f(\bL_{t+s}^{\sf LD})\,\rmd s + \sqrt{2}\;
		\Delta_h\bW_t \approx \bL_t^{\sf LD} - h \nabla f
		(\bL_t^{\sf LD}) + \sqrt{2}\;\Delta_{h}\bW_t. 
\end{align}
By repeatedly applying this approximation, we can 
construct the Markov chain $(\bvartheta^{\sf LMC}_k:k\in
\mathbb N)$ given by
\begin{align}\label{eq:LMC}
    \bvartheta_{k+1}^{\sf LMC} = \bvartheta_k^{\sf LMC} 
    -h\nabla f(\bvartheta_k^{\sf LMC})
    +\sqrt{2}\,(\bW_{(k+1)h}-\bW_{kh}).
\end{align}
Since $\bvartheta^{\sf LMC}_k\approx \bL^{\sf LD}_{kh}$, 
the distribution of $\bvartheta^{\sf LMC}_k$ is expected
to be close to the target $\pi$ as $k$ tends to 
infinity and $h$ goes to zero. 

Note that the Brownian motion in \eqref{eq:LMC} need not be identical to that used in the Langevin diffusion; an alternative Brownian motion with better coupling properties may be preferable. However, our proofs rely on synchronous sampling, requiring that discrete-time processes and their continuous-time counterparts use the same Brownian motion.  
Therefore, for the sake of consistency and clarity, we 
maintain the notation $\bW$ throughout this discussion, 
both in equations \eqref{eq:LMC} and \eqref{eq:LD}.

Let $U$ be a random variable uniformly distributed in
$[0,1]$ and independent of the Brownian motion $\bW$. 
The randomized midpoint method~\cite{shen2019randomized} exploits the approximation
\begin{align}
    \bL^{\sf LD}_{t+h}&=\bL^{\sf LD}_t - \int_0^h 
    \nabla f(\bL^{\sf LD}_{t+s})\, \rmd s + \sqrt{2}\;
		\Delta_h\bW_t \label{update:LD}\\
    &\approx \bL^{\sf LD}_t - h \nabla f(\bL^{\sf LD
		}_{t+hU}) + \sqrt{2}\;\Delta_h\bW_t,
        \label{update:LD2}
\end{align}
which has the advantage of being unbiased, in the sense that
the expectation with respect to $U$ of the expression 
in \eqref{update:LD2} equals the right-hand side of 
\eqref{update:LD}. 
Let us now introduce an alternative approach for approximating the integral $\int_0^h \nabla f(\bL_{t+s}^{\sf LD})\rmd s$, 
 as proposed in \citep{shen2019randomized}. 
This approach 
offers substantial reductions in running time through the 
parallelization of computations. The interval $[0,h]$ is 
divided into $R$ segments, each of length ${h}/{R}$.  
Choosing $U_r$ uniformly from the $r$-th interval $\big[
\frac{r-1}{R},\frac{r}{R}\big]$ for $r=1,\dots,R$, 
the approximation within each interval is given by
\begin{align}\label{eq:approx}
	\int_0^h  \nabla f(\bL_{t+s}^{\sf LD})\,\rmd s = 
	\sum_{r=1}^R \int_{\frac{r-1}{R}h}^{\frac{r}{R}h}  
	\nabla f(\bL_{t+s}^{\sf LD})\,\rmd s
	\approx \sum_{r=1}^R \frac{h}{R}\grad f(\bL^{\sf LD
	}_{t+hU_r}).
\end{align}
Assume that there exists an approximation $\bvartheta_k$ 
of $\bL^{\sf LD}_t$ at time $t = kh$.  To advance to $t+h = (k+1)h$ using relations \eqref{update:LD} and \eqref{eq:approx}, we need to approximate the sequence $(\bL^{\sf LD}_{t+hU_1},\ldots,\bL^{\sf LD}_{t+hU_R})$. We begin with a rough approximation $(\bvartheta_k^{(0,1)},\dots,\bvartheta_k^{(0,R)})=(\bvartheta_k,\ldots,\bvartheta_k)$, then perform $Q$ sequential steps to update this vector, each involving multiple parallel computations. At step $q$, the approximation $(\bvartheta^{(q,1)}_{k}, \ldots, \bvartheta^{(q,R)}_{k})$ depends only on the previous values $(\bvartheta^{(q-1,r)}_k)_{r=1}^R$, enabling parallelization.

Since, similar to the approximation in \eqref{eq:approx}, 
it holds that
\begin{align}
    L_{t+hU_r}^{\sf LD} 
		&= L_t^{\sf LD} - \sum_{j=1}^{r-1} \int_{\frac{j-1}{R} 
			h}^{\frac{j}{R}h}  \nabla f(\bL_{t+s}^{\sf LD})\,\rmd s  
			- \int_{\frac{r-1}{R} h}^{hU_r} \nabla f(\bL_{t+s}^{
			\sf LD})\,\rmd s + \sqrt{2}\,\Delta_{hU_r} \bW_t\\
    &\approx L_t^{\sf LD} -  \sum_{j=1}^{r-1} \frac{h}{R}
			\grad f(\bL^{\sf LD}_{t+hU_j}) -  h(U_r - \tfrac{r-1}{R}
			)\nabla f(\bL^{\sf LD}_{t+hU_r}) + \sqrt{2}\,
			\Delta_{hU_r} \bW_t,
\end{align}
the update rule for $\bvartheta_k^{(q,r)}$, inspired by the work of 
Shen and Lee \citep{shen2019randomized}, is 
defined as follows
\begin{align}
    \bvartheta^{(q,r)}_k = \bvartheta_k -  \sum_{j=1}^{r-1} 
		\frac{h}{R}\grad f(\bvartheta_k^{(q-1,j)}) -  h(U_r - 
		\tfrac{r-1}{R})\nabla f(\bvartheta_k^{(q-1,r)})
    + \sqrt{2}\,\Delta_{hU_r} \bW_t.
\end{align}
This leads to the parallelized version of RLMC, referred to 
as pRLMC, the formal definition of which is provided in~\Cref{RLMCp}. 
For simplicity, the superscript $\sf pRLMC$ is omitted therein. 
\Cref{fig:RLMCp} illustrates this algorithm.

\begin{algorithm}[ht]
\small
\caption{Parallelized RLMC}\label{RLMCp}
\noindent\textbf{Input}: number of parallel steps $R$, number of sequential iterations $Q$, step size $h$, number of iterations $n$, initial point~$\bvartheta_0.$\\
\textbf{Output}: iterate $\bvartheta^{}_{n+1}$
\begin{algorithmic}[1]
\For{\texttt{$k=1$ to $n$}}
   \State Draw $U_{kr}$ uniformly from $\big[\frac{r-1}{R},\frac{r}{R}\big]$, $r=1,\dots,R$.
   \State Generate $\bxi_{k}=\sqrt{2}(\bW_{(k+1)h}-\bW_{kh})$. 
   \State Generate $\bxi_{kr}=\sqrt{2}(\bW_{(k+U_{kr})h} -
   \bW_{kh})$, $r=1,\dots,R$.
   \State Set $\bvartheta_{k}^{(0,r)}=\bvartheta_k, r=1,\dots,R$.
   \For{\texttt{$q=1$~to~$Q-1$}}
    \For{\texttt{$r=1$ to $R$} in parallel} \hspace{4cm} $\tikzmark{listing-4-end}$
    \State
    $a_{kjr}= \min\{\frac{1}{R},U_{kr}-\frac{j-1}{R}\},j=1,
    \dots,r$ 
     \State $\bvartheta_{k}^{(q,r)}=\bvartheta_k-h\sum_{j=1}^r a_{kjr} \grad f\big(\bvartheta_{k}^{(q-1,j)}\big) 
     + \bxi_{kr}$.
     \EndFor $\tikzmark{listing-7-end}$
    \EndFor		
    \State  $\bvartheta_{k+1}=\bvartheta_k -\frac{h}{R}\sum_{r=1}^R \grad f\big(\bvartheta_{k}^{(Q-1,r)}\big)
    +\bxi_k$.
\EndFor	
\end{algorithmic}
\AddNote{listing-4-end}{listing-7-end}{listing-4-end}{In parallel}
\end{algorithm}

\begin{figure}[ht]
\centering
\includegraphics[width=0.75\linewidth]{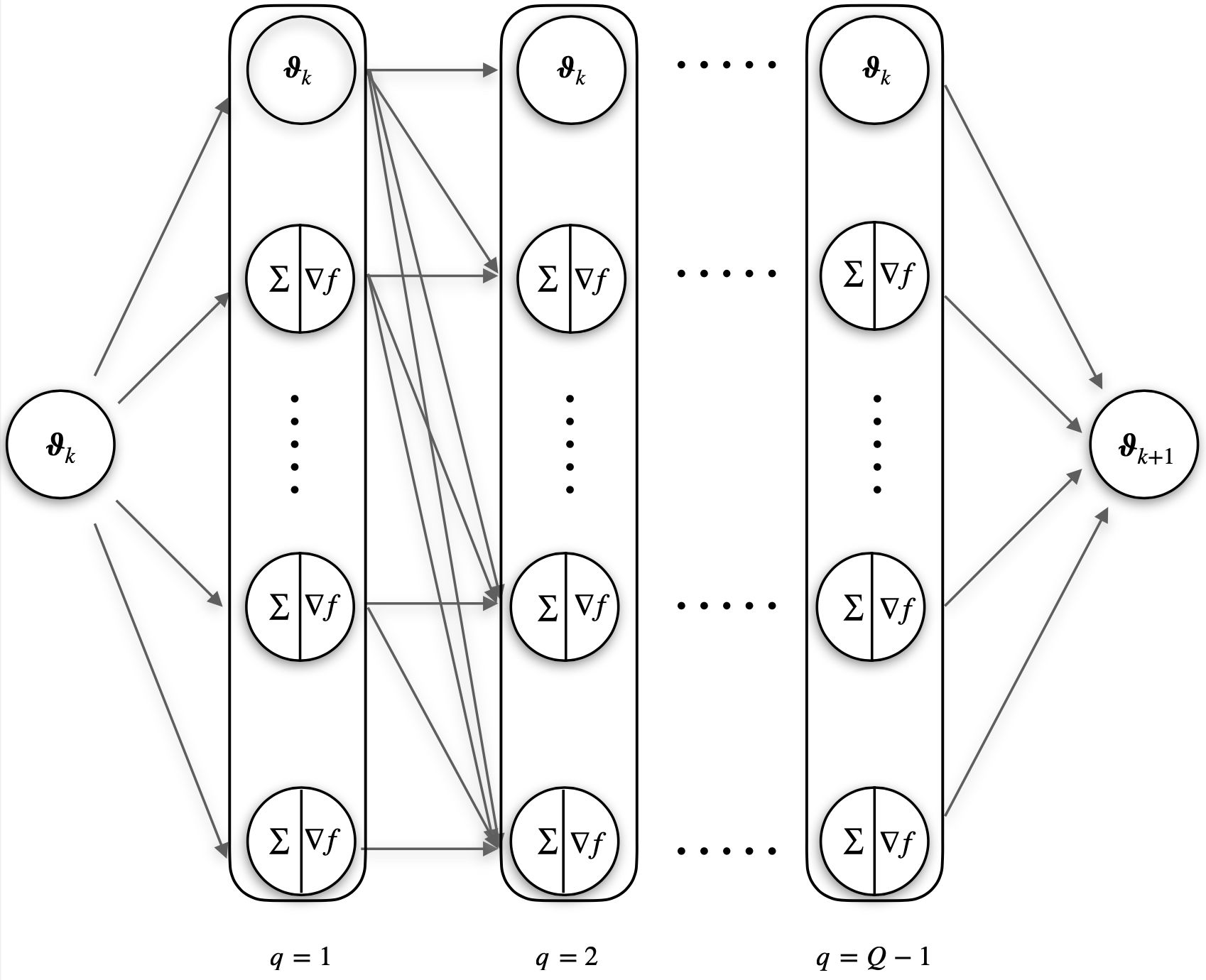}
   \caption{Visualization of Parallelized RLMC.}
   \label{fig:RLMCp}
\end{figure}

\normalsize

We state below the theoretical guarantee for the proposed 
parallelized RLMC algorithm. Proof of all the  results stated 
in this paper are deferred to the supplementary material.
\begin{theorem}
\label{thm:rlmc}
    If the function $f:\RR^p\to \RR$ satisfies  
    Assumption~\ref{asm:A-scgl} and the parameters $(h,Q,R)$ are 
    such that 
    $\bar h^{Q}+\bar h/R+(\bar h^{Q-1}+\bar h/R^{3/2})\sqrt{\kappa \bar h}\leqslant 0.1$ with $\bar h=Mh,$
    then, 
    for every $n\geqslant 1$, the 
    the distribution $\nu_n^{\textup{\sf pRLMC}}$ of 
    $\bvartheta_n^{\textup{\sf pRLMC}}$ satisfies
    \begin{align}\label{ineq:RLMC}
        \wass_2(\nu_n^{\sf pRLMC}, \pi)\leqslant 1.03e^{-mnh/2}
        \wass_2(\nu_0,\pi) + 2.1 \bigg(\bar h^{Q}+\frac{\bar h}
        {\sqrt{R}}+\Big(\bar h^{Q-1}+\frac{\bar h}{R}\Big)
        \sqrt{\kappa \bar h}\bigg) \sqrt{\frac{p}{m}}.
    \end{align}
\end{theorem}

This result recovers the convergence rate presented in Theorem 1 in~\cite{yu2023langevin} for the RLMC algorithm, with  $R=1$ and $Q=2$. 
The proof extends techniques from prior work on the RLMC algorithm~\cite{yu2023langevin}, while addressing new challenges introduced by the parallel setting. More details on the 
differences between our proof and the previous analysis can be found in~\ref{app:proof-rlmc}.

\section{Parallelized Randomized midpoint method for the 
kinetic Langevin diffusion}

The present section examines the application of parallelization techniques to the randomized midpoint method for kinetic Langevin processes (RKLMC), originally proposed by  \cite{shen2019randomized}. Our analysis focuses on two key aspects: quantifying the algorithm's accuracy and establishing bounds on its computational complexity. The kinetic Langevin process $\bL^{\textup{\sf KLD}}$ is characterized by a second-order stochastic differential equation, which can be informally expressed as
\begin{align}\label{KLD:1}
  {\textstyle\frac1{\gamma}}\ddot\bL_t^{\sf KLD} + 
  \dot\bL_t^{\sf KLD} = -\nabla f(\bL_t^{\sf KLD}) 
  + \sqrt{2}\,\dot\bW_t,
\end{align}
where the initial conditions are given by $\bL_0^{\sf KLD} = \bvartheta_0$ and $\dot\bL_0^{\sf KLD} = \bv_0$. Here, $\gamma$ is a positive parameter, $\bW$ denotes a standard $p$-dimensional Brownian motion, and dots represent derivatives 
with respect to time $t \geqslant 0$. As $\gamma$ approaches infinity, this process reduces to the (vanilla) Langevin process discussed in the previous section. A rigorous formulation of $\bL^{\sf KLD}$ can be achieved through It\^o's calculus by introducing the velocity field $\bV^{\sf KLD}$ as the time derivative of $\bL^{\sf KLD}$. The resulting joint process $(\bL^{\sf KLD}, \bV^{\sf KLD})$ evolves according to
\begin{align}\label{KLD:2}
    \rmd\bL^{\sf KLD}_t = \bV_t^{\sf KLD}\,\rmd t;
    \quad  
    \tfrac1{\gamma}\rmd\bV^{\sf KLD}_t = -\big(
		\bV_t^{\sf KLD} + \nabla f(\bL_t^{\sf KLD})\big)
		\,\rmd t + \sqrt{2}\, \rmd \bW_t.
\end{align}
Similar to the vanilla Langevin diffusion \eqref{eq:LD}, 
the kinetic Langevin diffusion $(\bL^{\sf KLD}, \bV^{\sf
KLD})$ is a Markov process that exhibits ergodic 
properties when the potential $f$ is strongly 
convex (see \citep{eberle2019} and references 
therein). The invariant density of this process 
is given by
\begin{align}
    p_*(\btheta_,\bv) \propto \exp\{-f(\btheta) - 
		{\textstyle\frac1{2\gamma}}\|\bv\|^2\}, \qquad
    \text{for all}\quad \btheta,\bv\in\mathbb R^p.
\end{align}
Note that the marginal of $p_*$ corresponding to 
$\btheta$ coincides with the target density $\pi$.

To obtain a discrete-time approximation suitable for sampling applications, Shen and Lee \cite{shen2019randomized} proposed the following procedure, termed RKLMC. For iterations $k = 1,2,\ldots$, the algorithm proceeds as:
\begin{enumerate}
\setlength\itemsep{0.01em}
    \item  generate random 
    vectors $(\bxi_k',\bxi''_k,\bxi_k''')$ and a random
    variable $U_k$ such that they are independent of all the variables
    generated at the previous steps and satisfy
    \begin{itemize}
    \setlength\itemsep{0.05em}
        \item $U_k$ is uniformly distributed in $[0,1]$,
        \item conditionally to $U_k = u$, $(\bxi_k',
				\bxi''_k,\bxi'''_k)$ has the same joint 
				distribution as 
            $\big(\bB_u -  \bG_u,
                \bB_1 - \bG_1,
                \gamma \bG_1
            \big)$,
        where $\bB$ is a Brownian motion in $\mathbb R^p$
        and $\bG_t = \int_0^{t} e^{\gamma h(s-t)}\,\rmd 
        \bB_s$.
    \end{itemize}
    \item set $\psi(x) = (1-e^{-x})/x$ and define the 
    $(k+1)$-th iterate of $\bvartheta^{\sf RKLMC}$ by
    \begin{align}
        \bvartheta_{k+U} &=\bvartheta_k+  
        U_k h \psi(\gamma U_k h) \bv_k - {U_k h} \big(1 - \psi 
        (\gamma U_k h)\big)\nabla f(\bvartheta_k) + 
        \sqrt{2h}\,\bxi'_k\\
        \bvartheta_{k+1} &= \bvartheta_k +  h
        \psi(\gamma h)\bv_k -  \gamma h^2( 1 - U_k)
        \psi\big(\gamma h(1-U_k)\big) \nabla f 
        (\bvartheta_{k+U}) + \sqrt{2h} \bxi''_k\\
        \bv_{k+1} &= e^{-\gamma h}\bv_k - {\gamma} h
        e^{- \gamma h(1 -U_k)} \nabla f(\bvartheta_{k+U}) 
        +  \sqrt{2h}\,\bxi'''_k.
    \end{align}
\end{enumerate}
The RKLMC algorithm fundamentally relies on sequential gradient evaluations of $f$, with each iteration requiring two evaluations that must be computed one after the other. Shen and Lee \cite{shen2019randomized} also developed a modified version, termed pRKLMC, which overcomes this sequential constraint through a parallel implementation. For reference, we present the complete pRKLMC algorithm in Algorithm~\ref{RKLMCp}, where we suppress the superscript $\sf pRKLMC$ for notational simplicity. 
\begin{algorithm}[ht]
\small
\caption{Parallelized  RKLMC}\label{RKLMCp}
\noindent\textbf{Input}: number of parallel steps $R$, 
number of sequential iterations $Q$, step size $h$, 
friction coefficient~$\gamma$, number of iterations $n$, 
initial points~$\bv_0$ and $\bvartheta_0.$\\
\textbf{Output}: iterates $\bvartheta^{}_{n+1}$ and 
$\bv_{n+1}$
\begin{algorithmic}[1]
\setstretch{1.20}
\For{\texttt{$k=1$ to $n$}}
   \State Draw $U_{k1},\dots,U_{kR}$ uniformly from 
		$\big[0,\frac{1}{R}\big],\dots,\big[\frac{R-1}{R}, 
		1\big],$ respectively.
   \State Generate  $\bxi_{k}=\sqrt{2}\int_0^{h}\big(1 
	- e^{-\gamma (h-s)}\big)d\overline{\bW}_{s}$, for
        $\overline{\bW}_s = \bW_{kh+s}-\bW_{kh}$.
   \State Generate  $\overline{\bxi}_{k} = \sqrt{2} \int_0^{h} 
        e^{-\gamma (h-s)}d\overline{\bW}_{s}$.
   \State Generate  $\bxi_{kr}=\sqrt{2}\int_0^{h}\big(1 
	- e^{\gamma U_{kr}(s-h)}\big)d\overline{\bW}_{s U_{kr}}$, for $r=1,\dots,R$. 
   \State Set $\bvartheta_{k}^{(0,r)}=\bvartheta_k, r=1,\dots,R$.
   \For{\texttt{$q=1$~to~$Q-1$}}
   \For{\texttt{$r=1$ to $R$} in parallel} \hspace{5cm} $\tikzmark{listing-4-end}$

    \State
    $a_{kr}=({1-e^{-\gamma hU_{kr}}})/{\gamma}.$
    \State
    $b_{kj}=\int_{{(j-1)h}/{R}}^{(h/R)\min({j},RU_{kr})} (1-e^{-\gamma(U_{kr}h-s)})ds, j=1,\dots,r.$
     \State  $\bvartheta_{k}^{(q,r)}=\bvartheta_k
     + a_{kr}\bv_k-\sum_{j=1}^r b_{kj}\grad f\big(\bvartheta_{k}^{(q-1,j)}\big) + \bxi_{kr}.$
     
     \EndFor $\tikzmark{listing-7-end}$
    \EndFor		
    \State  $ \bvartheta_{k+1}=\bvartheta_k 
    +\frac{1-e^{-\gamma h}}{\gamma}\bv_k
    -\sum_{r=1}^R  \frac{h}{R}(1-e^{-\gamma h(1-U_{kr})})
		\grad f\big(\bvartheta_{k}^{(Q-1,r)}\big)
    + {\bxi_k}$.
    \State  $\bv_{k+1}=e^{-\gamma h}\bv_k- \gamma\sum_{r=1}^R
		\frac{h}{R}e^{-\gamma h(1-U_{kr})}\grad f\big(\bvartheta_{k}^{(Q-1,r)}\big)
    + \gamma\overline{\bxi}_k$.
\EndFor	
\end{algorithmic}
\AddNote{listing-4-end}{listing-7-end}{listing-4-end}{In parallel}
\end{algorithm}

\normalsize

The following theorem establishes an upper bound on the $\wass_2$ distance between the sampling distribution generated by the pRKLMC algorithm and the target distribution. 

\begin{theorem}
\label{thm:rklmc}
Let $f:\RR^p\to \RR$ be a function satisfying  
Assumption~\ref{asm:A-scgl} and let $\btheta^*$ be a minimizer
of $f$. 
Choose
the parameters $(\gamma,R,Q,h)$ so that $\gamma \geqslant 5M$,
and  $\kappa\big(\bar h^6/R^3+ \bar h^{4Q-2}\big)
\leqslant 10^{-6},$ where $\bar h = \gamma h$. 
Assume that $\bvartheta_0$ is independent of $\bv_0$ and that
$\bv_0\sim\mathcal N_p(0,\gamma\bfI_p)$. Then, for any 
$n \geqslant 1$, the distribution $\nu_n=\nu_n^{\textup{\sf 
pRKLMC}}$ of $\bvartheta_n^{\textup{\sf pRKLMC}}$ satisfies
\begin{align}
    \wass_2(\nu_n,\pi) &
    \leqslant 
    3.04e^{-mnh}\wass_2(\nu_0,\pi) + 1.1\Big({\frac{e^{-mnh}}{m} \E[f(\bvartheta_0) - f(\btheta^*)] }\Big)^{1/2}\\
    &\qquad  + 80.11 \Big(\frac{\bar h^3}{R^2}+\bar h^{2Q-1} 
    \Big)^{1/2}\sqrt{\frac{p}{m}} + 4.33 \Big(\frac{\bar h^6}{R^3} 
    + \bar h^{4Q-2} \Big)^{1/2}\sqrt{\frac{\kappa p}{m}} \,.
\end{align}
\end{theorem}

To verify consistency, setting $Q=2$ and $R=1$ in this result recovers the convergence rate for RKLMC established in~\cite[Theorem 2]{yu2023langevin}. Furthermore, in the theoretical limit where $R$ approaches infinity—--corresponding to the ability to compute an infinite number of parallel gradient evaluations per unit time—--the pRKLMC error exhibits exponential decay at a rate proportional to $n\wedge Q$.

Theorem \ref{thm:rklmc} establishes the strongest known convergence rate for pRKLMC when sampling from a strongly log-concave distribution with a gradient Lipschitz potential. Our analysis reveals several key advances beyond the seminal work of Shen and Lee \cite{shen2019randomized}, where pRKLMC was first introduced. While~\cite{shen2019randomized} focused on determining the number of iterations needed to achieve an $\varepsilon$-bounded error, 
our work provides a comprehensive upper bound on the sampling error at any iteration. We accomplish this by extending the proof techniques from \cite{yu2023langevin} to address the challenges of the parallel setting. For a detailed comparison between our proof of \Cref{thm:rklmc} and its sequential counterpart in \cite{yu2023langevin}, see~\ref{app:proof-rklmc}. Notably, our bound in \Cref{thm:rklmc} features explicit, small constants and eliminates the need to initialize the algorithm at the potential's minimizer. These improvements may prove valuable for future extensions to non-convex potentials.

Moreover, our results highlight the trade-off between the number of sequential steps, $nQ$, and the parallel evaluations of the potential's gradient at each step, $R$. This trade-off becomes particularly significant when considering the choice proposed by~\cite{shen2019randomized} of $R$ proportional to $\ceil{ \sqrt{\kappa}/ \epsilon}$, which may prove impractical under either one of the following two conditions: first, when high-precision requirements demand a very small $\epsilon$, and second, when the potential function exhibits severe ill-conditioning. Let us also mention right away that our result shows that it
suffices to choose $R$ of a lower order than $\sqrt{\kappa}/\varepsilon$ 
to achieve the desired precision. Indeed, we will show in \Cref{cor:2} that
$R = \lceil (10/\varepsilon) + (\sqrt{\kappa}/\varepsilon)^{2/3}\rceil$ 
parallel evaluations of the gradient are sufficient to achieve this goal.

The computational challenges extend beyond parallel processing capabilities. Substantial storage constraints emerge even with sufficient parallel processing units to execute $\ceil{\sqrt{\kappa}/\epsilon}$ operations simultaneously. The pRKLMC algorithm requires maintaining $R$ vectors of dimension $p$ at each iteration. When both the dimensionality $p$ and the scaling factor $\sqrt{\kappa}/\epsilon$ are large, the memory requirements can quickly become prohibitive, potentially rendering the approach computationally infeasible. 
We further elaborate on this point in the next section.

As this manuscript approached its final stages, the paper by Anari et al.~\citep{anari2024fast} was posted on arXiv. The results presented in our work were independently derived and nearly concurrently with those in \citep{anari2024fast}, without prior knowledge of its findings. While both papers aim to accelerate Langevin sampling by parallelizing certain operations, they offer different contributions, as we elaborate
below.
The main strengths of \citep{anari2024fast} compared to our results are:
\begin{itemize}\itemsep=0pt
\item[\textbf{a)}] A weaker condition on the target density: our requirement of $m$-strong convexity of $f$ is replaced in \citep{anari2024fast} by the log-Sobolev inequality with constant $m$.
\item[\textbf{b)}] Capability to handle approximate evaluations of the potential's gradient.
\item[\textbf{c)}] In the case of vanilla Langevin diffusion, error bounds 
are established using three distinct metrics: the Kullback-Leibler divergence, total-variation and Wasserstein distances.
\end{itemize}
Regarding these strengths, several observations can be made:
\textbf{a)} Extending our proof to target densities beyond log-concave distributions that satisfy the log-Sobolev inequality appears to be very challenging, 
\textbf{b)} While extending our results to inexact gradient evaluations seems feasible, such an extension would necessitate tedious computational refinements, and 
\textbf{c)} Deriving bounds on the Kullback-Leibler and total variation distances for randomized midpoint methods that improve upon their vanilla counterparts remains a highly 
non-trivial open problem.

Conversely, our primary contribution compared to \citep{anari2024fast} lies in deriving Wasserstein distance upper bounds with significantly more favorable dependencies. Specifically, our bounds demonstrate improved scaling 
with respect to
\vspace*{-7pt}
\begin{itemize}\itemsep=-1pt
    \item[--]  the step size $h$ and the condition number 
    in the case of LMC,
    \item[--] the initialization error and the number of 
    iterations count, in the case of KLMC.
\end{itemize}

\vspace*{-9pt}
A more concrete and quantitative comparison is presented in
the next section.

We also note that \cite{anari2024fast} primarily relies on bounding the Kullback-Leibler error and characterizing the Langevin process distribution dynamics as a solution to a gradient flow problem in the space of measures, following the approach established in \citep{vempala2019rapid,durmus2019analysis}. It remains unclear how these proof techniques can be extended to exploit the potential advantages of randomized midpoint discretization. As a consequence, the algorithms investigated in \citep{anari2024fast}—which we will henceforth refer to as pLMC and pKLMC—do not incorporate the randomized midpoint step, in contrast to the pRLMC and pRKLMC algorithms analyzed in our study.

\section{Discussion}
\label{sec:comparison}

Our main theorems, stated in the preceding sections, address the question posed in the introduction: What computational resources are necessary to ensure that the generated random vector's distribution falls within a prescribed distance of the target distribution? In this section, we explore these insights under two distinct computational paradigms: first, in scenarios with an unlimited number of parallel processing units, and second, in settings constrained by a limited number of parallel units. We then discuss how these insights compare
to the results recently obtained in \citep{anari2024fast}.

\subsection{Consequences in the setting of an unlimited number of parallel units}

Consider a computational scenario where, for any given $R \in \mathbb{N}$, we can construct a device capable of performing $R$ parallel evaluations of the gradient of $f$ within a single unit of time. In this context, \Cref{thm:rlmc,thm:rklmc} guide for selecting $R$ to ensure that the upper bounds on sampling error are below the prescribed threshold $\varepsilon\sqrt{p/m}$. The corresponding results, obtained as direct consequences of \Cref{thm:rlmc,thm:rklmc}, respectively, read as follows. 

\begin{corollary}[Mixing time for the vanilla Langevin with parallelization and randomized midpoint discretization]\label{cor:1}
Let $\varepsilon \in (0,1)$ be fixed. Set 
\[
    R = \left\lceil {0.28(1+\varepsilon \sqrt{\kappa})/\varepsilon^2}
    \right\rceil, 
    \quad 
    Q = \left\lceil 2.2 + 0.7\log(\sqrt{\kappa}/\varepsilon)\right\rceil,
\]
and choose the step size $h > 0$ and the number of outer
iterations $n \in \mathbb{N}$ such that
\[
    Mh = 0.1
    \quad \text{and} \quad
    n \geqslant 10\kappa \Big\{\log\left({7}/{\varepsilon}\right) + \log\Big(\frac{m}{p}\wass_2^2(\nu_0,\pi)\Big)\Big\}.
\]
Then\footnote{This follows from the fact that $10\kappa = 2\kappa/(Mh) = 2/mh$.}, after $nQ$ iterations, each involving $R$ parallel queries of the gradient information $\nabla f$, 
the parallelized version of the randomized midpoint Langevin Monte-Carlo algorithm satisfies
\[
    \wass_2(\nu_n^{\sf pRLMC}, \pi) \leqslant \varepsilon\sqrt{{p}/{m}}.
\]
\end{corollary}

\begin{corollary}[Mixing time for the kinetic Langevin with 
parallelization and randomized midpoint discretization]\label{cor:2}
    Let $\varepsilon\in(0,1)$ be fixed. Set $\gamma = 5M$, $R =
    \ceil{10\varepsilon^{-1} + \kappa^{1/3} 
    \varepsilon^{-2/3}}$, $Q = 5 + \ceil{0.7\log (\sqrt{\kappa} 
    /\varepsilon)}$, $\bvartheta_0 = \btheta^*$ and choose 
    the step size $h>0$ and the number of outer iterations 
    $n\in\mathbb N$ such that 
    \begin{align}
     \gamma h = 0.1,
    \quad
    \text{and} 
    \quad
        n\geqslant 25 \kappa\Big\{ \log(22/\varepsilon)+
        \log\Big(\frac{m}{p}\wass_2^2(\nu_0,\pi)\Big)\Big\}
        ,
    \end{align}
    then
    we have $\wass_2(\nu_n^{\sf pRKLMC}, 
    \pi) \leqslant \varepsilon\sqrt{p/m}$ after $nQ$ iterations, 
    each of which involves $R$ parallel queries of the gradient 
    information $\grad f$.
\end{corollary}

These corollaries show that the number of sequential steps, $nQ$, 
sufficient to achieve an error bounded by $\varepsilon 
\sqrt{p/m}$ is of order $\kappa \log(1/\varepsilon)
\log(\kappa/\varepsilon)$ both for the vanilla and the kinetic
versions of the algorithms. However, the number of gradient 
evaluations carried out at each iteration is much lower in the
case of the kinetic Langevin diffusion based algorithm: 
$\varepsilon^{-1} + (\sqrt{\kappa}/\varepsilon)^{2/3}$
versus $\varepsilon^{-2} + \sqrt{\kappa}/\varepsilon$. 
It should be noted that this rate is better than the rates
of $R$ obtained in the original work \cite{shen2019randomized} 
and in the concurrent paper \cite{anari2024fast}. Indeed, both
papers obtain $R$ of order of $\sqrt{\kappa}/\varepsilon$. 
Therefore, our results improves the rate of $R$ by a factor
of the order $\kappa^{-1/2} + (\varepsilon\kappa^{-1/2})^{1/3}$.

\subsection{Consequences in the setting of a limited 
number of parallel units}

Let us now examine a more practical scenario where we have a 
fixed number $R$ of gradient evaluations that can be performed 
in parallel. The following result builds on \Cref{thm:rlmc} and 
\Cref{thm:rklmc}, simplifying their claims by omitting 
numerical constants. This makes the results both easier to 
interpret and more readily comparable with existing literature.
Thus, we use the notation \(a_n \lesssim b_n\) to indicate that 
the ratio \(a_n / b_n\) is bounded above by a universal constant. 
Similarly, \(a_n \gtrsim b_n\) is equivalent to \(b_n \lesssim a_n\), 
and \(a_n \asymp b_n\) indicates that both \(a_n \lesssim b_n\) 
and \(b_n \lesssim a_n\) hold.

\begin{corollary}
\label{cor:3} 
Let $f$ be $m$-strongly convex and $M$-gradient Lipschitz. Set 
$\bar h = Mh$ and assume that the initial value $\bvartheta_0$ 
is equal to the minimizer of $f$, then 
\begin{align}
    \wass_2^2(\nu_n^{\sf pRLMC},\pi) &\lesssim \bigg(e^{-n\bar h/\kappa} 
	  + \frac{\bar h^3 \kappa}{R^2}
	 +\frac{\bar h^2}{R}\bigg)\cdot\frac{p}{m}\,
        \label{up:pRLMC}\\
    \wass_2^2(\nu_n^{\sf pRKLMC},\pi) 
	& \lesssim \bigg(e^{-n\bar h/\kappa} + 
        \frac{\bar h^6 \kappa}{R^3} + \frac{\bar 
        h^3}{R^2}\bigg)\cdot\frac{p}{m}\,,\label{up:pRKLMC}
\end{align}
provided that $\bar h = Mh\lesssim 1$, $Q\gtrsim 1+ \log R$ and 
$\gamma = 5M$. Therefore, to ensure that these upper bounds remain below $\varepsilon^2{p/m}$, we must select $n$ as follows
\begin{align}
  \text{\sf pRLMC}&: n\asymp \kappa\log(1/\varepsilon)\Big(1 +   
    \Big\{\frac{\kappa}{R^2\varepsilon^2}\Big\}^{1/3} + \Big\{\frac{1}{R\varepsilon^{2}}\Big\}^{1/2}\Big)\,,\label{n:pRLMC}\\
    \text{\sf pRKLMC}&: n\asymp \kappa\log(1/\varepsilon)\Big(1 +   
    \Big\{\frac{\kappa}{R^3\varepsilon^2}\Big\}^{1/6}
    +\Big\{\frac{1}{R\varepsilon}\Big\}^{2/3}\Big)\,.\label{n:pRKLMC}
\end{align}
\end{corollary}

To compare our results with Anari et al.~\cite{anari2024fast}, we first express their findings in the same format as our corollary above. For their parallel variants of Langevin Monte-Carlo (pLMC) and Kinetic Langevin Monte-Carlo (pKLMC), the bounds are:
\begin{align}
    \label{up:pLMC}
    \wass_2^2(\nu_n^{\sf pLMC},\pi) &\lesssim \bigg(
		e^{-n\bar h/\kappa} +  \frac{\kappa \bar h}{ R}\bigg)
		\cdot\frac{p}{m}\\
    \label{up:pKLMC}
    \wass_2^2(\nu_n^{\sf pKLMC},\pi) &\lesssim  \bigg(e^{-n\bar h/15\kappa}
		(1+\log \kappa) + \frac{n \bar h^3\kappa }{R^4} + 
            \frac{n\bar h^3}{R^2}\bigg)\cdot\frac{p}m\ .
\end{align}
Bound \eqref{up:pLMC} holds when $Mh\lesssim 1$, $\kappa Mh\lesssim \sqrt{R}$, and $Q\gtrsim 1+\log R$ (this follows from the final display in the proof of Theorem~13 in~\cite{anari2024fast}). While bound~\eqref{up:pLMC} appears explicitly in~\cite{anari2024fast}, bound~\eqref{up:pKLMC} is derived from their work. Specifically, they prove a similar bound for the TV-distance in Theorems~15, 20, and~21, and their analysis provides all the necessary components to establish~\eqref{up:pKLMC}.

Comparing bounds~\eqref{up:pRLMC} and~\eqref{up:pLMC} reveals that the randomized midpoint discretization achieves a better discretization error. The improvement, of order $\frac{\bar h^2}{R} + \frac{\bar h}{\kappa}$, can be substantial. A similar advantage appears in the kinetic Langevin diffusion algorithms: comparing~\eqref{up:pRKLMC} and~\eqref{up:pKLMC}, and noting that $n$ should be at least of order $\kappa/\bar h$, we find an improvement of order $\frac{\bar h^3}{R} + \frac{1}{\kappa}$. 

From bounds~\eqref{up:pLMC} and~\eqref{up:pKLMC}, we obtain analogs to relations~\eqref{n:pRLMC} and~\eqref{n:pRKLMC}:
\begin{align+}
    \textsf{pLMC}&: n\gtrsim \kappa \log(1/\varepsilon) \Big(1 + 
    \frac{\kappa}{\sqrt{R}} + \frac{\kappa}{R\varepsilon^2}\Big),
        \label{n:pLMC}\\
    \textsf{pKLMC}&: n\gtrsim \kappa\log(1/\varepsilon) \Big(1 +
    \Big\{\frac{\kappa}{R^2\varepsilon^2}\Big\}^{1/2} + \frac{\kappa}{R^2\varepsilon}\bigg)\log(1+\log\kappa).\label{n:pKLMC}
\end{align+}
These bounds demonstrate two key insights: first, increasing $R$ leads to a polynomial reduction in the required number of iterations, highlighting the benefits of parallelization. Second, when compared to~\eqref{n:pRLMC} and~\eqref{n:pRKLMC}, they confirm the advantages of the randomized midpoint discretization.

\subsection{Time complexity versus computational complexity}

Parallelization primarily aims to reduce time complexity by executing multiple computations simultaneously. Most often, this reduction in time complexity comes at the cost of increased computational complexity. In the context of sampling with first-order oracles, as considered here, algorithms achieving optimal time complexity might require more gradient queries than those optimized for computational complexity. Unfortunately, the algorithms considered here are no exception to this rule.  For example, the Langevin Monte Carlo algorithm has a computational complexity of order $(\kappa/\varepsilon^2)\log(1/\varepsilon)$ \citep{durmus2019analysis}. In contrast, the parallelized LMC version proposed in \citep{anari2024fast}, when optimized for time complexity, has a total computational complexity 
$nQR$ of order at least $(\kappa^2/\varepsilon^2)\log(1/\varepsilon)$. 
A consequence of our results is that, for the parallelized randomized midpoint Langevin algorithms (pRLMC and pRKLMC), optimizing for time complexity incurs only a moderate increase in computational complexity. Specifically, \Cref{cor:1} and \Cref{cor:2} establish that the number of gradient queries required for pRLMC and pRKLMC are of order 
\((\kappa/\varepsilon^2 + \kappa^{3/2}/\varepsilon) \log(1/\varepsilon) 
\log(\kappa/\varepsilon)\) and \((\kappa/\varepsilon + \kappa^{4/3} 
/\varepsilon^{2/3})\log(1/\varepsilon)\log(\kappa/\varepsilon)\), respectively. Combined with the rates established in \citep{shen2019randomized,he2020ergodicity,yu2023langevin}, 
this implies that the computational overhead is of order 
\[
\frac{\kappa/\varepsilon^2 + \kappa^{3/2}/\varepsilon}{\kappa/\varepsilon + \kappa^{4/3}/\varepsilon^{2/3}} \lesssim 1/\varepsilon +  (\sqrt{\kappa}/\varepsilon)^{1/3}
\]
in the case of the vanilla Langevin, and 
\[
\frac{\kappa/\varepsilon + \kappa^{4/3}/\varepsilon^{2/3}}{\kappa/\varepsilon^{2/3} + \kappa^{7/6}/\varepsilon^{1/3}} \lesssim (1/\varepsilon)^{1/3} + (\sqrt{\kappa}/\varepsilon)^{1/3}
\]
in the case of kinetic Langevin. These results demonstrate that 
kinetic Langevin diffusion-based algorithms achieve a favorable trade-off between time complexity and computational overhead.

\section{Acknowledgements}
This work was partially supported by the center \href{https://www.hi-paris.fr/}{Hi! PARIS} and the grant Investissements 
d’Avenir (ANR-11-IDEX0003/Labex Ecodec/ANR-11-LABX-0047).

 \bibliographystyle{elsarticle-num} 
 \bibliography{cas-refs}

\appendix
\section{Proof of the upper bound on the error of parallelized RLMC}\label{app:proof-rlmc}

This section is devoted to the proof of the upper bound on the error of parallelization of RLMC.
Without any risk of confusion, we will use the notation $\bvartheta_k$ instead
of $\bvartheta^{\sf pRLMC}_k$ to refer to the $k$-th iterate of the
RLMC. We will also use the shorthand notation
\begin{align}
    f_k: = f(\bvartheta_k),\qquad 
   \bg_k = \grad f(\bvartheta_k),\qquad 
   \bg_k^{(q,r)}=\grad f(\bvartheta_k^{(q,r)}),\qquad
    \text{and}\qquad \nabla f_{k+U}:=\grad f(\bvartheta_{k+U}).
\end{align}
{Moreover, throughout this section, we define $\delta=h/R,\bar h=Mh.$}

\textbf{Proof overview}: 
Our proof follows the strategy of Theorem~1 in~\cite{yu2023langevin}, which focuses on bounding two key elements: the mean discretization error and the deviation of each iterate $\bvartheta_k$. However, the parallel setting of our work requires significant adaptations to this approach. While~\cite{yu2023langevin} analyzes these bounds over a single interval $[0,h]$, we must consider $R$ separate time windows $[\frac{(r-1)h}{R}, \frac{rh}{R}]$, where $r = 1,\dots,R$. The corresponding bounds for the discretization error and deviation are established in \Cref{lem:A2} and \Cref{lem:A1}, respectively.

\begin{proof}[Proof of Theorem~\ref{thm:rlmc}]

    Let $\bvartheta_k\sim\nu_k$ be the state of the pRLMC algorithm 
    at the $k$th iteration, and let $\bL_0\sim\pi$ be a random vector 
    in $\mathbb R^p$ defined on the same probability space such that 
    $\wass_2(\nu_k,\pi) = \mathbb E[\|\bvartheta_k-\bL_0\|_2^2]$. We 
    assume that the probability space is sufficiently large to allow 
    the definition of a Brownian motion $\bar\bW$, independent of $(\bvartheta_k,\bL_0)$, and the sequence of iid random variables, uniformly distributed in $[0,1]$, $U_{k1},\ldots, U_{kR}$, 
    independent of $(\bvartheta_k,\bL_0,\bar\bW)$. 
    We define the Langevin diffusion 
    \begin{align}\label{eq-int-lang}
    		\bL_t = \bL_0 - \int_{0}^{t}\nabla f(\bL_s)\,\rmd s + \sqrt{2}\, \bar\bW_t,\qquad t\in[0,h].
    \end{align}
    We can then define the random variables $\bxi_k$ and $(\bxi_{kr}
    )_{1\leqslant r\leqslant R}$ of the steps 3 and 4 of 
    \Cref{RLMCp} by $\bxi_k = \sqrt{2}\,\bar\bW_h$ and 
    $\bxi_{kr} = \sqrt{2}\,\bar\bW_{hU_{kr}}$, so that a version
    of $\bvartheta_{k+1}$ is given by $\bvartheta_{k+1} = 
    \bvartheta_k - \frac{h}{R}\sum_{r=1}^R \nabla f(\bvartheta_k^{
    (Q-1,r)}) + \bxi_k$. 

    We set $x_k = \wass_2(\nu_k,\pi)$ and note that
    \begin{align}
        x_{k+1} \leqslant \mathbb E[\|\bvartheta_{k+1}- \bL_{h}\|_2^2] := \|\bvartheta_{k+1}- \bL_{h} \|_{\Ltwo}^2. 
    \end{align}
    We will also consider the Langevin process on the time interval $[0,h]$ given by 
    \begin{equation}\label{eq-int-lang1}
		\bL'_t = \bL'_0 - \int_{0}^{t}\nabla f(\bL'_s)\,\rmd s + \sqrt{2} \,\bar\bW_{t},\qquad \bL'_0 = \bvartheta_k. 
    \end{equation}
    Note that the Brownian motion is the same as in \eqref{eq-int-lang}. 
   Let us introduce one additional notation, the average of 
    $\bvartheta_{k+1}$ with respect to $U_{k1},\dots,U_{kR}$,
    \begin{equation}
        \bar\bvartheta_{k+1} = \mathbb E[\bvartheta_{k+1} | 
        \bvartheta_k,\bar\bW,\bL_0]. 
    \end{equation}
    Since $\bL_{h}$ is independent of $U_{k1},\dots,U_{kr}$, it is clear that 
    \begin{align}
        x_{k+1}^2 
        & \leqslant \|\bvartheta_{k+1} - \bar\bvartheta_{k+1} 
        \|_{\Ltwo}^2 + \|\bar\bvartheta_{k+1} - \bL_{kh} 
        \|_{\Ltwo}^2.
    \end{align}
    Furthermore, the triangle inequality yields
    \begin{align}
        \|\bar\bvartheta_{k+1} - \bL_{h} \|_{\Ltwo}
        &\leqslant \|\bar\bvartheta_{k+1} - \bL'_{h} \|_{\Ltwo} +
        \|\bL'_h - \bL_{h} \|_{\Ltwo}.
    \end{align}
    From the exponential ergodicity of the Langevin diffusion \cite{bhattacharya1978}, we get
    \begin{align}
        \|\bL'_h-\bL_{h}\|_{\Ltwo} \leqslant e^{-mh}\|\bL'_0-\bL_{0}\|_{\Ltwo} = e^{-mh}\|\bvartheta_k - \bL_{0}
        \|_{\Ltwo} = e^{-mh}x_k.
    \end{align}
    Therefore, we get
    \begin{align}
        x_{k+1}^2 
        & \leqslant \|\bvartheta_{k+1} - \bar\bvartheta_{k+1} 
        \|_{\Ltwo}^2 + \big(\| \bar\bvartheta_{k+1} -\bL'_h
        \|_{\Ltwo} + e^{-mh}x_k \big)^2\,.
        \label{eq:4.4}
    \end{align}
We will need the following lemmas, the proofs of which are postponed.
\begin{lemma}\label{lem:A2}
    Let $bar h = Mh$ and $\delta = h/R$. 
    When $\bar h \leqslant 0.1$ and $R\geqslant 2$, it holds that 
    \begin{align}
        \|\bar\bvartheta_{k+1}-\bL'_h\|_{\Ltwo} 
        &
        \leqslant {\bar h^{Q}}\big(0.56 h \|\bg_k \|_{\mathbb L_2} +
        1.05\sqrt{h p}\big) + \frac{2\bar h^2}{3R}  \big(\sqrt{h\delta} \| \bg_k\|_{\Ltwo}  + 2 \sqrt{hp}\big).
    \end{align}
\end{lemma}
\begin{lemma}\label{lem:A1}
    Let $bar h = Mh$ and $\delta = h/R$. 
    When $\bar h \leqslant 0.1$ and $R\geqslant 2$, it holds that 
\begin{align}
    \|\bvartheta_{k+1} - \bar\bvartheta_{k+1} \|_{\Ltwo}
    &\leqslant \bar h^Q(0.56h \|\bg_k\|_{\mathbb L_2} + 1.05 
    \sqrt{hp}) + 
    0.7\, \bar h\big(\delta \|\bg_k\|_{\mathbb L_2}  + 
    1.7\sqrt{\delta p}\big)\,.
\end{align}
\end{lemma}

    One can check by induction that if for some $A\in [0,1]$ 
    and for two positive sequences $\{B_k\}$ and $\{C_k\}$
    the inequality $x_{k+1}^2\leqslant \big\{(1-A)x_k + C_k
    \big\}^2 + B_k^2$ holds for every integer $k\geqslant 0$,
    then\footnote{This is an extension of 
    \citep[Lemma 7]{dalalyan2019user}. It essentially relies 
    on the elementary $\sqrt{(a+b)^2 + c^2}\leqslant a+ 
    \sqrt{b^2 + c^2}$, which should be used to prove the 
    induction step.} 
    \begin{align}\label{eq:xk}
        x_n\leqslant (1 - A)^n x_0 + \sum_{k=0}^n (1-A)^{n-k} 
        C_k +\bigg\{\sum_{k=0}^n (1-A)^{2(n-k)}B_k^2\bigg\}^{1/2}
    \end{align}
In view of \eqref{eq:xk}, \eqref{eq:4.4}, \Cref{lem:A2},
 and \Cref{lem:A1}, for $\rho = e^{-mh}$, we get
 {
 \begin{align}
        x_n&\leqslant \rho^{n} x_0 + \sum_{k=0}^n 
        \rho^{n-k} \Big\{ \bar h^{Q} (0.56h\|\bg_k\|_{\Ltwo} + 
				1.05\sqrt{hp} )  + \frac23\bar h^2\Big(\frac{h}{R^{3/2}}\|\bg_k\|_{\Ltwo} +\frac{2\sqrt{hp}}{R}
                \Big)\Big\}\\
    & \quad  + \Bigg\{\sum_{k=0}^n\rho^{2(n-k)} 
    \Big[ \bar h^Q(0.56h \|\bg_k\|_{\mathbb L_2} + 1.05 
    \sqrt{hp}) + 0.7\, \bar h\big(\delta \|\bg_k\|_{\mathbb L_2}  + 
    1.7\sqrt{\delta p}\big) \Big]^2\Bigg\}^{1/2}\\
    &\leqslant \rho^n x_0 + 0.56 h\bar h^Q \sum_{k=0}^n \rho^{n-k}\|\bg_k\|_{\Ltwo} + 1.11 \bar h^Q\frac{\sqrt{hp}}{mh} 
    + \frac{2\bar h^2\,h}{3R^{3/2}} \sum_{k=0}^n \rho^{n-k}
    \|\bg_k\|_{\Ltwo} \\
    &\quad    + 1.41 \bar h^2\frac{\sqrt{hp}}{R}\frac{1}{mh}
    + 0.56 h \bar h^{Q}\bigg\{ \sum_{k=0}^n \rho^{2(n-k)} 
    \|\bg_k\|_{\Ltwo}^2 \bigg\}^{1/2} + 0.58 \bar h^{Q} 
    \sqrt{\frac{p}{m}} \\
    &\quad + 0.7\frac{\bar h\, h}{R}\bigg\{ \sum_{k=0}^n \rho^{2(n-k)}
    \|\bg_k\|_{\Ltwo}^2 \bigg\}^{1/2} +  0.66\frac{\bar h}{\sqrt{R}}\sqrt{\frac{p}{m}} \,.
        \label{eq:xn-bound}
\end{align}
}
The last display follows from the fact that $1-\rho
\geqslant 0.95 mh$. Let us set $S_n(g^2) = 
\sum_{k=0}^n\rho^{n-k} \|\bg_k\|_{\Ltwo}^2$. By leveraging 
the inequality $1-\rho \geqslant 0.95 mh$ and invoking 
the Cauchy–Schwarz inequality, we obtain
\begin{align}
 \bar h \sum_{k=0}^n \rho^{n-k}\|\bg_k\|_{\Ltwo}
 & \leqslant \bigg(\bar h^{2}\sum_{k=0}^n\rho^{n-k}\sum_{k=0}^n\rho^{n-k} \|\bg_k\|_{\Ltwo}^2\bigg)^{1/2}
 \leqslant 1.03 \sqrt{\kappa M h S_n(g^2)}\,.
\end{align}
Plugging these into display~\eqref{eq:xn-bound} then gives 
{
\begin{align}
    x_n  &\leqslant \rho^{n} x_0  +  h\bigg(0.6 \bar h^{Q} + 0.7
    \frac{\bar h}{R} + 0.6 \bar h^{Q-1}\sqrt{\kappa Mh} + 
    0.7 \frac{\bar h}{R^{3/2}}\sqrt{\kappa Mh}\bigg) \sqrt{S_n(g^2)}\\
    &\qquad +  \Big(1.11\bar h^{Q-1}\sqrt{\kappa Mh}+1.41\frac{\bar h}{R}\sqrt{\kappa Mh} +0.6\bar h^Q+ \frac{0.7\bar h}{\sqrt{R}}\Big)  \sqrt{{p}/{m}}.
\end{align}}
In what follows, let us use the following notation
\begin{align}
    \Diamond = \bar h^{Q}+\frac{\bar h}{{R}}+(\bar h^{Q-1}+\bar h/R^{3/2})\sqrt{\kappa Mh},\quad 
    \triangle = \bar h^{Q}+\frac{\bar h}{\sqrt{R}}+(\bar h^{Q-1}+\bar h/R)\sqrt{\kappa Mh}\,.
\end{align}
in order to simplify writing. With this in mind, we can rewrite the last 
display as
\begin{align}
    x_n 
    &\leqslant \rho^{n} x_0 
    + 0.7 h \Diamond \sqrt{S_n(g^2)}+  1.41 \triangle \sqrt{{p}/{m}}
    \label{eq:xn-bound1}
\end{align}
We will need the following lemma for finding an upper bound 
on the right-hand side of the last display.
\begin{lemma}\label{lem:A3} 
    If $Mh\leqslant 0.1$, the following 
    inequality holds
    \begin{align}
        h^2 S_n(g^2) = 
        h^2\sum_{k=0}^n \rho^{n-k} \|\nabla f ( 
        \bvartheta_k)\|^2_{\Ltwo} 
        &\leqslant 0.17
        \rho^n\|\bvartheta_0-\btheta^*\|_{\Ltwo}^2+ 0.94 (p/m)\,,
    \end{align}    
    where $\bvartheta^*$ denotes the minimizer of the function $f$.
    \end{lemma}
The claim of this lemma together with \eqref{eq:xn-bound1} entails that
\begin{align}
x_n
& \leqslant \rho^n x_0
+0.7 \Diamond
\Big(0.42\rho^{n/2}\|\bvartheta_0 - \btheta^*\|_{\Ltwo}+0.97\sqrt{{p}/{m}} \Big)  + 1.41 \triangle \sqrt{{p}/{m}}\,.
\end{align}
Rearranging the display and noting that $\|\bvartheta_0 - \btheta^*\|_{\Ltwo}\leqslant \wass_2(\nu_0,\pi)+\sqrt{p/m}$ and $\Diamond\leqslant \triangle$, 
we arrive at 
\begin{align}
x_n & \leqslant \rho^n x_0 + 0.3 \Diamond e^{-mnh/2}
    \|\bvartheta_0-\btheta^*\|_{\Ltwo}  +  2.1 \triangle 
    \sqrt{{p}/{m}}\\
    & \leqslant \rho^n x_0 +0.3 \Diamond e^{-mnh/2}
    \wass_2(\nu_0,\pi) +  2.1 \triangle \sqrt{{p}/{m}}\\
    & \leqslant \big[1+ 0.3 \Diamond
    \big] e^{-mnh/2} \wass_2(\nu_0,\pi) + 2.1 
    \triangle \sqrt{{p}/{m}}\,.
\end{align}
When $\Diamond \leqslant 0.1,$ we obtain $
x_n\leqslant 1.03e^{-mnh/2}\wass_2(\nu_0,\pi)
+ 2.1 \triangle \sqrt{{p}/{m}}$ 
as desired.

\end{proof}

\subsection{Proof of technical lemmas}

In this section, we present proofs for the technical lemmas that underpin \Cref{thm:rlmc}. The first lemma provides an upper bound on the error between the averaged iterate $\bar\bvartheta_{k+1}$ and the continuous-time diffusion $\bL'$, which originates from $\bvartheta_k$ and evolves until time $h$. 
The second lemma establishes an upper bound on the distance separating the iterate $\bvartheta_{k+1}$ from its averaged counterpart $\bar\bvartheta_{k+1}$. Both of these upper bounds contain the norm of the potential's gradient evaluated at $\bvartheta_k$. The third lemma derives an upper bound for 
the discounted sums of squared gradient norms.

We first need the following auxiliary lemma.
{
\begin{lemma}\label{lem:A0}
    Let $bar h = Mh$ and $\delta = h/R$. 
    When $\bar h \leqslant 0.1$ and $R\geqslant 2$, it holds that 
\begin{align}
    \frac{1}{R}\sum_{r=1}^R  \big\|
    \bvartheta_{k}^{(Q-1,r)}-  \bL'_{hU_{kr}} \big\|_{\Ltwo}
    &\leqslant {\bar h^{Q-1}}\big\{0.56 h \|\bg_k \|_{\mathbb L_2} +
    0.74\sqrt{2h p}\big\} + \frac{2M\delta}3  \big\{\sqrt{h\delta} \| \bg_k\|_{\Ltwo}  + 2 \sqrt{hp}\big\}.
\end{align}
\end{lemma}
\begin{proof}
Let $q\in\{1,2,\ldots,Q-1\}$. 
It was proved in Lemma 1 in~\cite{yu2023langevin} 
that for every $t\in[0,h]$,
\begin{equation}
\label{eq:prev11}
\|\bL'_t-\bL'_0\|_{\mathbb L_2}
\leqslant e^{Mt}\big({t^2\|\bg_k\|_{\Ltwo}^2+2tp}\big)^{1/2}
\leqslant 1.106\, t\|\bg_k\|_{\Ltwo} + 1.106\sqrt{2tp}.
\end{equation}
This implies that for every $t,s\in[0,h]$ such that $s<t$,
\begin{align}
    \|\bL'_t-\bL'_s\|_{\mathbb L_2}
    &   \leqslant \int_{s}^t \|\nabla f(\bL'_u)\|_{\mathbb L_2}\rmd u
    + \sqrt{2p(t-s)}\\
    &\leqslant (t-s)\|\bg_k\|_{\mathbb L_2} + M\int_s^t 
    \big(1.106\, u\|\bg_k\|_{\Ltwo} + 1.106\sqrt{2up}\big)\,\rmd u + 
    \sqrt{2(t-s)p}\\
    &\leqslant 1.12(t-s)\|\bg_k\|_{\mathbb L_2} + 1.08\sqrt{2(t-s)p}.
    \label{eq:prev12}
\end{align}
By the definition of $ \bvartheta_{k}^{(q,r)}$, taking into
account that $ha_{kjr} = \int_{(j-1)\delta}^{(j\delta)\wedge  
(hU_{kr})}1\rmd s\leqslant \delta$ and using the triangle 
inequality, we obtain
\begin{align}
    \Big\|\bvartheta_{k}^{(q,r)}-  \bL'_{hU_{kr}} \Big\|
	&= \bigg\|\sum_{j=1}^{r-1} \delta \bg_k^{(q-1,j)} + 
            \big(hU_{kr} - (r-1) \delta\big) \bg_k^{(q-1,r)} -
            \int_0^{hU_{kr}}\!\!\!\!\!\grad f(\bL'_{s})\,\rmd s 
            \bigg\| \\
        & \leqslant \bigg\|\sum_{j=1}^{r-1} \delta \big(
            \bg_k^{(q-1,j)} - \nabla f(\bL'_{hU_{kj}})\big)\bigg\|
            + \bigg\|\sum_{j=1}^{r-1}\int_{(j-1)\delta}^{j\delta}
            \!\!\big(\grad f(\bL'_{s}) - \grad 
            f(\bL'_{hU_{kj}}) \big)\,\rmd s\bigg\| \\
        &\qquad + \delta \big\|\big( \bg_k^{(q-1,r)} - \nabla f
            (\bL'_{hU_{kr}})\big)\big\| + \bigg\|\int_{(r-1) 
            \delta}^{hU_{kr}} \big(\nabla f(\bL'_s) - \nabla f(\bL'_{hU_{kr}})\big) \rmd s\bigg\|,
\end{align}
where for getting the last inequality we have split the integral 
over $[0, hU_{kr}]$ into the sum of integrals over the partition
$\{[0,\delta],[\delta,2\delta],\ldots, [(r-1)\delta, U_{kr}]\}$.
Using the Lipschitz property  of $\nabla f$, we get
\begin{align}
    \Big\|\bvartheta_{k}^{(q,r)}-  \bL'_{hU_{kr}} \Big\|
	&= \bigg\|\sum_{j=1}^{r-1} \delta \nabla f(
            \bvartheta_k^{(q-1,j)}) + \big(hU_{kr} - (r-1)
            \delta\big) \nabla f(\bvartheta_k^{(q-1,r)}) -
            \int_0^{hU_{kr}}\!\!\!\!\!\grad f(\bL'_{s})\,\rmd s 
            \bigg\| \\
        & \leqslant M\delta\sum_{j=1}^{r} \big\|
            \bvartheta_k^{(q-1,j)} - \bL'_{hU_{kj}}\big\|
            + \bigg\|\sum_{j=1}^{r-1}\int_{(j-1)\delta}^{j\delta}
            \!\!\big(\grad f(\bL'_{s}) - \grad 
            f(\bL'_{hU_{kj}}) \big)\,\rmd s\bigg\| \\
        &\qquad + M \int_{(r-1) 
            \delta}^{hU_{kr}} \big\|\bL'_s - \bL'_{hU_{kr}}\big\| 
            \rmd s.\label{lem4:eq1}
\end{align}
One easily checks that $Z_j:= \int_{(j-1)\delta}^{j\delta} \big(
\grad f(\bL'_{s}) - \grad f(\bL'_{hU_{kj}}) \big)\,\rmd s$ are independent
random variables with zero mean. Therefore,
\begin{align}
    \bigg\|\sum_{j=1}^{r-1}Z_j\bigg\|_{\mathbb L_2}^2
    &= \sum_{j=1}^{r-1}  \bigg\|\int_{(j-1)\delta}^{j\delta}
    \!\!\big(\grad f(\bL'_{s}) - \grad 
    f(\bL'_{hU_{kj}}) \big)\,\rmd s\bigg\|_{\mathbb L_2}^2\\
    &\leqslant M^2\delta\sum_{j=1}^{r-1}  \int_{(j-1)\delta}^{j\delta}
    \!\!\big\|\bL'_{s} - \bL'_{hU_{kj}} \big\|_{\mathbb L_2}^2\,\rmd s\\
    &\leqslant M^2\delta\sum_{j=1}^{r-1}  \int_{(j-1)\delta}^{j\delta}
    \bigg\{\big\|\int_{hU_{kj}}^s \nabla f(\bL'_{u})\rmd u\big\|_{\mathbb L_2}^2 + 2p\mathbb E[|hU_{kj}-s|^2]\bigg\}\,\rmd s\\
    &\leqslant M^2\delta\sum_{j=1}^{r-1}  \int_{(j-1)\delta}^{j\delta}
    \bigg\{\big\|\int_{hU_{kj}}^s \nabla f(\bL'_{u})\rmd u\big\|_{\mathbb L_2}^2 + 2p\delta^2/3\bigg\}\rmd s\\
    & \leqslant  M^2\delta(r-1) \Big\{\frac{\delta^3}6 \max_{v\in[0,h]}
    \|\nabla f(\bL'_{v})\|_{\mathbb L_2}^2  + \frac{2p\delta^2}3\Big\}.
\end{align}
The first inequality follows from the Cauchy-Schwarz Inequality. 
The first term of the last line above can be upper bounded using 
\eqref{lem:prev} as follows:
\begin{align}
    \big\| \nabla f(\bL'_{v})\big\|_{\Ltwo}
    &\leqslant  \| \bg_k\|_{\Ltwo} + \big\| 
    \nabla f(\bL'_{v}) - \nabla f(\bL'_{0})\big\|_{\Ltwo}
    \leqslant \| \bg_k\|_{\Ltwo} + M\big\| \bL'_{v} - 
    \bL'_{0}\big\|_{\Ltwo}\\
    &\leqslant \| \bg_k\|_{\Ltwo} + 1.106 Mv \|\bg_k\|_{\mathbb L_2} 
    + 1.106 M \sqrt{2vp}\\
    &\leqslant 1.12\| \bg_k\|_{\Ltwo}  + 0.35 \sqrt{2Mp},
\end{align}
where we have used the fact that $Mv\leqslant Mh\leqslant 0.1$. 
Therefore, 
\begin{align}
    \bigg\|\sum_{j=1}^{r-1}\int_{(j-1)\delta}^{j\delta}
            \!\!\big(\grad f(\bL'_{s}) &- \grad 
            f(\bL'_{U_{kj}}) \big)\,\rmd s\bigg\|_{\mathbb L_2} = 
    \bigg\|\sum_{j=1}^{r-1}Z_j\bigg\|_{\mathbb L_2}\\
    & \leqslant  M\sqrt{\delta (r-1)} \bigg\{\frac{\delta^{3/2}}{\sqrt{6}}
    (1.12\| \bg_k\|_{\Ltwo}  + 0.35 \sqrt{2Mp})  + \sqrt{\frac{2p\delta^2}3}\Big\}\\
    &\leqslant M\delta\sqrt{r-1}\big(0.46\delta\| \bg_k\|_{\Ltwo}  + 0.9 \sqrt{\delta p}\big).\label{lem4:eq2}
\end{align}
Similarly, in view of \eqref{eq:prev12},
\begin{align}
    \bigg\|\int_{(r-1) 
        \delta}^{hU_{kr}} \big\|\bL'_s - \bL'_{hU_{kr}}\big\| \,
        \rmd s\bigg\|_{\mathbb L_2}^2 &\leqslant  \frac1\delta
        \int_0^\delta \bigg(\int_0^u \|\bL'_{(r-1)\delta + u} -
        \bL'_{(r-1)\delta + v}\|_{\mathbb L_2}\,\rmd v\bigg)^2\,\rmd u \\
        &\leqslant  \frac1\delta
        \int_0^\delta \Big(\int_0^u\big\{1.12 |u-v| \|\bg_k\|_{\mathbb L_2} 
        + 1.08\sqrt{2|u-v|p}\big\}\,\rmd v\Big)^2\,\rmd u\\
        & = \frac1\delta
        \int_0^\delta \Big(0.56 u^2 \|\bg_k\|_{\mathbb L_2} 
        + 0.72 u\sqrt{2 u p}\Big)^2\,\rmd u\\
        &\leqslant \delta^2\Big(0.33\delta\|\bg_k\|_{\mathbb L_2} + 
        0.42\sqrt{2\delta p}\Big)^2.
\end{align}
Combining \eqref{lem4:eq1}, \eqref{lem4:eq2} and the last 
display, we get
\begin{align}
    \Big\|\bvartheta_{k}^{(q,r)}-  \bL'_{hU_{kr}} \Big\|_{\mathbb L_2}
	& \leqslant M\delta\sum_{j=1}^{r} \big\|
            \bvartheta_k^{(q-1,j)} - \bL'_{hU_{kj}}\big\|_{\mathbb L_2}
            + M\delta \sqrt{r-1} \big(0.46\delta\| \bg_k\|_{\Ltwo} 
            + 0.9 \sqrt{\delta p}\big)\\
        &\qquad + M \delta\big\{0.33\delta \| \bg_k\|_{\mathbb L_2} 
        + 0.42 \sqrt{2\delta p}\big\}.
\end{align}
Since $M\delta = Mh/R = \bar h/R$, it then follows that
\begin{align}
    \frac1R \sum_{r=1}^R\Big\|\bvartheta_{k}^{(q,r)} - \bL'_{hU_{kr}} 
    \Big\|_{\mathbb L_2}
	& \leqslant \frac{\bar h}{R}\sum_{j=1}^{R} \big\|
            \bvartheta_k^{(q-1,j)} - \bL'_{hU_{kj}}\big\|_{\mathbb L_2}
            + (2/3)M\delta \sqrt{R} \big(0.46\delta\| \bg_k\|_{\Ltwo} 
            + 0.9 \sqrt{\delta p}\big)\\
        &\qquad + M \delta\big\{0.33\delta \|
            \bg_k\|_{\mathbb L_2} + 0.42\sqrt{2\delta p}\big\}\\
        &\leqslant \frac{\bar h}{R}\sum_{j=1}^{R} \big\|
            \bvartheta_k^{(q-1,j)} - \bL'_{hU_{kj}}\big\|_{\mathbb L_2}
            +  0.6M\delta  \big\{\sqrt{h\delta} \| \bg_k\|_{\Ltwo}  
            + 2 \sqrt{hp}\big\}.
\end{align}
By induction over $q$, this inequality entails that
\begin{align}
    \frac1R \sum_{r=1}^R\Big\|\bvartheta_{k}^{(Q-1,r)} - \bL'_{hU_{kr}} 
    \Big\|_{\mathbb L_2}
	& \leqslant \frac{\bar h^{Q-1}}{R}\sum_{j=1}^{R} \big\|
            \bvartheta_k - \bL'_{hU_{kj}}\big\|_{\mathbb L_2}
            +  \frac23 M\delta  \big\{\sqrt{h\delta} \| \bg_k\|_{\Ltwo}  
            + 2 \sqrt{hp}\big\}\\
        & = \frac{\bar h^{Q-1}}{R}\sum_{j=1}^{R} \big\|
            \bL_0' - \bL'_{hU_{kj}}\big\|_{\mathbb L_2}
            +  \frac23 M\delta  \big\{\sqrt{h\delta} \| \bg_k\|_{\Ltwo}  
            + 2 \sqrt{hp}\big\}.
\end{align}
Using once again \eqref{eq:prev11}, we get 
\begin{align}
    \frac{\bar h^{Q-1}}{R}\sum_{j=1}^{R} \big\|
            \bL_0' - \bL'_{hU_{kj}}\big\|_{\mathbb L_2} &\leqslant
            \frac{\bar h^{Q-1}}{R}\sum_{j=1}^{R} 1.106\big\{(j\delta)
            \|\bg_k \|_{\mathbb L_2} + \sqrt{2j\delta p}\big\}\\
            &\leqslant
            {\bar h^{Q-1}}\big\{0.56 h \|\bg_k \|_{\mathbb L_2} +
            0.74\sqrt{2h p}\big\}.
\end{align}
This completes the proof of the lemma. 
\end{proof}
We now proceed with the proof of \Cref{lem:A2}. Throughout the following proofs, we denote by $\mathbb{E}_U$ the expectation where integration is performed with respect to the random variables $(U_{ki})_{1\leqslant i\leqslant R}$ only. In other words, $\mathbb{E}_U$ represents the conditional expectation given all random variables except $(U_{ki})_{1\leqslant i\leqslant R}$. 
\begin{proof}[Proof of \Cref{lem:A2}]
We note that
    \begin{align}
        \|\bar\bvartheta_{k+1}-\bL'_h\|_{\Ltwo} & =
        \bigg\|\bvartheta_{k}-\mathbb E_U\Big[\sum_{i=1}^R \delta \grad f (\bvartheta_k^{(Q-1,i)})\Big]
        -\bL'_0 +\int_0^h \nabla f(\bL'_s)\,\rmd s \bigg\|_{\Ltwo}\\
        & = \bigg\| \sum_{i=1}^R \delta \mathbb E_{U}\Big[\nabla f(\bvartheta_{k}^{(Q-1,i)}) - \grad f(\bL'_{hU_{ki}}) \Big] \bigg\|_{\Ltwo}
        \leqslant  \frac{\bar h}{R}\sum_{i=1}^R  \Big\|
        \bvartheta_{k}^{(Q-1,i)}-  \bL'_{hU_{ki}} \Big\|_{\Ltwo} \,.
        \label{eq:1}
    \end{align}    
Plugging \Cref{lem:A0} into 
the previous display then provides us with
\begin{align}
    \|\bar\bvartheta_{k+1}-\bL'_h\|_{\Ltwo}
    \leqslant {\bar h^{Q}}\big\{0.56 h \|\bg_k \|_{\mathbb L_2} +
    0.74\sqrt{2h p}\big\} + \frac{2\bar h^2}{3R}  \big\{\sqrt{h\delta} \| \bg_k\|_{\Ltwo}  + 2 \sqrt{hp}\big\}.
\end{align}
as desired.
\end{proof}

\begin{proof}[Proof of \Cref{lem:A1}]

By the definition of $\bvartheta_{k+1},\delta,$ and the fact that the mean minimizes the squared integrated error, we find
\begin{align}
\|\bvartheta_{k+1} - \bar\bvartheta_{k+1} \|_{\Ltwo}
    & =  \bigg\| \delta\sum_{r=1}^R\bg_k^{(Q-1,r)}- \delta\sum_{r=1}^R\E_U[\bg_k^{(Q-1,r)}] \bigg\|_{\Ltwo} \\
    & \leqslant  \delta\sum_{r=1}^R \big\| \bg_k^{(Q-1,r)}
    - \E_U[\bg_k^{(Q-1,r)}] \big\|_{\Ltwo} \\
    & \leqslant  \delta\sum_{r=1}^R\big\|
    \bg_k^{(Q-1,r)} - \grad f(\bL'_{(r-1)\delta}) 
    \big\|_{\Ltwo} \\
    &\leqslant \frac{\bar h}{R}\sum_{r=1}^R \big\| 
    \bvartheta_k^{(Q-1,r)}-\bL'_{(r-1)\delta} \big\|_{\Ltwo} \,.
\label{eq:helper2}
\end{align}
To obtain the upper bound for the display~\eqref{eq:helper2}, 
we derive an upper bound for $\big\|\bvartheta_k^{(Q,r)}
-\bL'_{(r-1)\delta}\big\|_{\Ltwo}$. On the one hand, the triangle 
inequality yields
\begin{align}
    \big\|\bvartheta_k^{(Q-1,r)} - \bL'_{(r-1)\delta}\big\|_{\Ltwo}
    &\leqslant \big\|\bvartheta_k^{(Q-1,r)}-\bL'_{hU_{kr}} 
    \big\|_{\Ltwo} + \big\|\bL'_{hU_{kr}} - \bL'_{(r-1)\delta}
    \big\|_{\Ltwo}.\label{A.y1}
\end{align}
On the other hand, in view of \eqref{eq:prev12}, for a random 
variable $U$ uniformly distributed in $[0,1]$, we have
\begin{align}
    \big\|\bL'_{hU_{kr}} - \bL'_{(r-1)\delta}
    \big\|_{\Ltwo}
    &\leqslant  \E\big[\big(1.12\delta U\|\bg_k\|_{\mathbb L_2} 
    + 1.08\sqrt{2U\delta p}\big)^2\big]^{1/2}\\
    &\leqslant 1.12\delta \|\bg_k\|_{\mathbb L_2}/\sqrt{3} + 1.08\sqrt{\delta p}\,.\label{A.y2}
\end{align}
Therefore, putting together \eqref{eq:helper2}, \eqref{A.y1} and 
\eqref{A.y2} we arrive at
\begin{align}
\|\bvartheta_{k+1} - \bar\bvartheta_{k+1} \|_{\Ltwo}
    &\leqslant \frac{\bar h}{R}\sum_{r=1}^R \big\| 
    \bvartheta_k^{(Q-1,r)}-\bL'_{hU_{kr}} \big\|_{\Ltwo}  + 
    0.65 \delta \bar h\|\bg_k\|_{\mathbb L_2}  + 
    1.08\bar h\sqrt{\delta p}    \,.
\end{align}
Plugging \Cref{lem:A0} into the previous display then 
provides us with
\begin{align}
\|\bvartheta_{k+1} - \bar\bvartheta_{k+1} \|_{\Ltwo}
    &\leqslant \bar h^Q(0.56h \|\bg_k\|_{\mathbb L_2} + 0.74 
    \sqrt{2hp}) + \frac{2\bar h^2}{3R}(\sqrt{h\delta} 
    \|\bg_k\|_{\mathbb L_2} + 2\sqrt{h p})\\
    &\qquad + 
    0.65 \delta \bar h\|\bg_k\|_{\mathbb L_2}  + 
    1.08\bar h\sqrt{\delta p} \\
    & \leqslant \bar h^Q(0.56h \|\bg_k\|_{\mathbb L_2} + 1.05 
    \sqrt{hp}) + 
    0.7\, \bar h\big(\delta \|\bg_k\|_{\mathbb L_2}  + 
    1.7\sqrt{\delta p}\big)
\end{align}
and the claim of the lemma follows. 
\end{proof}

}

\begin{proof}[Proof of \Cref{lem:A3}] We need the following auxiliary lemma  from~\cite{yu2023langevin}.

\begin{lemma}[Lemma 3 in~\cite{yu2023langevin}]
\label{lem:prev}
    Suppose $\omega = (\omega_n)_{n\in \mathbb N}$ is a
    real valued sequence. Define $S_n 
    (\omega) := \sum_{k = 0}^{n} \varrho^{n-k} 
    \omega_{k}$ and $S_n^{+1} 
    (\omega) := \sum_{k = 0}^{n} \varrho^{n-k} 
    \omega_{k+1}$. Then, $
        S_n^{+1} (\omega) = \omega_{n+1} - \varrho^{n+1} 
        \omega_0 + \varrho S_n(\omega)$.
\end{lemma}
Recall the notation
$\bxi_k = \sqrt{2}\,\bar\bW_{h}$ and note that 
    \begin{align}
        f_{k+1} &\leqslant f_k + 
        \bg_k^\top (\bvartheta_{k+1} -
        \bvartheta_k) + \frac{M}{2}\|\bvartheta_{k+1} -
        \bvartheta_k\|_2^2\\
        &\leqslant f_k  -
        \bg_k^\top\bigg(\sum_{i=1}^R \delta \grad f (\bvartheta_k^{(Q-1,i)})  -
        \bxi_k  \bigg)+ \frac{M}{2}\bigg\|\sum_{i=1}^R \delta \grad f (\bvartheta_k^{(Q-1,i)})- \bxi_k\bigg\|^2\\
        & = f_k 
        +  \bg_k^\top\bigg( \sum_{i=1}^R \delta \big(\bg_k- \bg_k^{(Q-1,i)}\big) -h\bg_k\bigg)+   \bg_k^\top\bxi_k + \frac{M}{2} \bigg\|\delta \sum_{i=1}^R \bg_k^{(Q-1,i)} - \bxi_k\bigg\|^2\\
        &\leqslant f_k  -h \|\bg_k\|^2 + 
        \delta  \|\bg_k\| \sum_{i=1}^R\Big\|\bg_k - \bg_{k}^{(Q-1,i)}\Big\| + 
        \bg_k^\top\bxi_k + \frac{M}{2}\bigg\|\delta\sum_{i=1}^R \bg_k^{(Q-1,i)}- \bxi_k\bigg\|^2
        \label{eq:n1}
    \end{align}
    We first derive the bound for $\sum_{i=1}^R \big\|\bg_k 
    - \bg_{k}^{(Q-1,i)}\big\|_{\mathbb L_2}$. For every $q\in\mathbb N$, one can check that
    \begin{align}
    \big\|\bg_k - \bg_{k}^{(q,i)}\big\|_{\mathbb L_2}
    & \leqslant  M\big\|\bvartheta_k - \bvartheta_{k}^{(q,i)}\big\|_{\mathbb L_2} 
    =  M\bigg\| \sum_{j=1}^i ha_{kjr} \grad f(\bvartheta_k^{(q-1,j)})
    -\bxi_{ki}\bigg\|_{\mathbb L_2} \\
    &\leqslant M \sum_{j=1}^i ha_{kjr} \big\|\bg_k^{(q-1,j)}
    \|_{\mathbb L_2}  + M\|\bxi_{ki}\|_{\mathbb L_2}\\
    &\leqslant M\delta \sum_{j=1}^R  \big\|\bg_k^{(q-1,j)} 
    -\bg_k\big\|_{\mathbb L_2}
    + Mi\delta \|\bg_k\|_{\mathbb L_2} + M\sqrt{(2i-1)\delta p}\,.
    \end{align}
Since $M\delta  = {\bar h}/{R}$, the last inequality
implies
\begin{align}
\frac{1}{R} \sum_{j=1}^R\big\|\bg_{k}^{(q,j)}-\bg_k\big\|_{\mathbb L_2}
&\leqslant  \frac{\bar h}{R} \sum_{j=1}^R  \big\|  \bg_k^{(q-1,j)}-\bg_k\big\|_{\mathbb L_2}
    + \frac{R+1}{2R} \bar h\|\bg_k\|_{\mathbb L_2} + M\Big(\frac23 + \frac1R\Big)\sqrt{2-\frac1{R}}\sqrt{hp}\\
    &\leqslant  \frac{\bar h}{R} \sum_{j=1}^R  \big\|  \bg_k^{(q-1,j)}-\bg_k\big\|_{\mathbb L_2}
    + \frac{3}{4} \bar h\|\bg_k\|_{\mathbb L_2} + 
    \frac{7}{12} M\sqrt{6hp}.
\end{align}
When $R\geqslant 2$ and $\bar h\leqslant 1/10$, unfolding 
the previous recursive inequality, we arrive at
\begin{align}
\frac{1}{R} \sum_{j=1}^R\big\|\bg_{k}^{(Q-1,j)}-\bg_k\big\|_{\mathbb L_2}
&\leqslant \frac{10}{9}\Big(\frac{3}{4}\bar h\|\bg_k\|_{\mathbb L_2} + \frac{7M}{12} \sqrt{6hp}\Big)
= \frac{5}{6} Mh\|\bg_k\|_{\mathbb L_2} + \frac{70M}{108} \sqrt{6hp}.
\label{eq:diff}
\end{align}
Hence,
\begin{align}
     \delta \|\bg_k\|_{\mathbb L_2} \sum_{i=1}^R\big\|\bg_{k}^{(Q-1,i)}-\bg_k\big\|_{\mathbb L_2} 
     &\leqslant \frac{5}{6} Mh^2 \|\bg_k\|^2_{\mathbb L_2}
     + {h}\Big( \frac{h\|\bg_k\|^2 }{2} + \big(\frac{70}{108}\big)^2 3p\Big)\\
     &\leqslant \frac{4}{3} h^2 \|\bg_k\|^2_{\mathbb L_2}
     + \frac43h p.
     \label{eq:3}
\end{align}
Furthermore, we have
\begin{align}
    \bigg\|\delta\sum_{i=1}^R \bg_k^{(Q-1,i)} - \bxi_k\bigg\|_{\mathbb L_2}
    &\leqslant\sum_{i=1}^R \delta\big\|\bg_k^{(Q-1,i)}-\bg_k\big\|_{\mathbb L_2}  + \| h\bg_k-\bxi_k\big\|_{\mathbb L_2} \\
    &\leqslant \delta\sum_{i=1}^R \big\|\bg_k^{(Q-1,i)} - \bg_k
    \|_{\mathbb L_2} + \sqrt{h^2\|\bg_k\|_{\mathbb L_2}^2 + 
    \| \bxi_k\big\|_{\mathbb L_2}^2}\\
    &
    \leqslant \frac{h}{R}  \sum_{j=1}^R\big\|\bg_{k}^{(Q-1,j)}-\bg_k\big\|_{\mathbb L_2} + \sqrt{h^2\|\bg_k\|_{\mathbb L_2}^2 + 2hp}.
    \end{align}
Combining this with the display~\eqref{eq:diff} gives
\begin{align}
\frac{M}{2}\bigg\|\delta \sum_{i=1}^R \bg_k^{(Q-1,i)} - 
\bxi_k\bigg\|^2_{\mathbb L_2}
&\leqslant \frac{M}{2}\Big( \frac56 Mh^2\|\bg_{k}\|_{\mathbb L_2} + 
\frac{70}{108} Mh\sqrt{6hp}
+ \sqrt{h^2\|\bg_k\|^2_{\mathbb L_2}+2hp}\Big)^2\\
&\leqslant \frac{M}{2}\Big( \frac1{12} h\|\bg_{k}\|_{\mathbb L_2} + \frac{7}{108}\sqrt{6hp}
+ \sqrt{h^2\|\bg_k\|^2_{\mathbb L_2}+2hp}\Big)^2\\
&\leqslant \frac{M}{2}\Big( \sqrt{(\tfrac1{12})^2 + 3(\tfrac{7}{108})^2}+1\Big)^2(h^2\|\bg_k\|^2+2hp)\\
&\leqslant \frac23 Mh^2 \|\bg_k\|^2+ \frac43 Mhp
\,.
\label{eq:4}
\end{align}
Plugging the displays~\eqref{eq:3} and~\eqref{eq:4} back into the display~\eqref{eq:n1} and taking the expectation, we get
\begin{align}\label{eq:n2}
\mathbb E[f_{k+1}] &\leqslant 
\mathbb E[f_k] - h\|\bg_k\|_{\Ltwo}^2 +  \frac43 Mh^2
\|\bg_k\|_{\Ltwo}^2 + \frac43Mhp + \frac23 Mh^2 \|\bg_k\|_{\Ltwo}^2
+\frac43 Mhp\\
&\leqslant 
\mathbb E[f_k] - 0.8 h\|\bg_k\|_{\Ltwo}^2 
+ (8/3)Mhp.
\end{align}
    
Let $f^* = \min_{\btheta\in\mathbb R^p} f(\btheta)$.  Set $S_n(f) = \sum_{k=0}^n \rho^{n-k} \mathbb E[f_k-f^*]$, 
$S_n(\bvartheta^2) = \sum_{k=0}^n \rho^{n-k} \|
    \bvartheta_k-\btheta^*\|_{\Ltwo}^2$ and 
    $S_n(g^2) = \sum_{k=0}^n \rho^{n-k} \|
    \bg_k\|_{\Ltwo}^2$. 
Using \Cref{lem:prev}, we get
    \begin{align}
        \mathbb E[f_{n+1} -f^*] - \rho^{n}\mathbb E[f_0] 
        + \rho S_n(f) \leqslant S_n(f) - 0.8 h
        S_n(\bg^2) + \frac{(8/3) Mhp}{1-\rho}
    \end{align}
    Since $mh\geqslant 1-\rho\geqslant 0.95 mh$, we get
    \begin{align}
        0.8 h
        S_n(\bg^2) &\leqslant \rho^{n}\mathbb E[f_0-f^*] 
        + (1-\rho)  S_n(f)  + {2.81 \kappa p}\\
        &\leqslant \rho^{n}\mathbb E[f_0-f^*] 
        + mh  S_n(f)  + {2.81  \kappa p}\\
        &\leqslant \rho^{n}\mathbb E[f_0-f^*] 
        + 0.5 h  S_n(\bg^2)  + {2.81  \kappa p}\,,
    \end{align}
    where the last line follows from the Polyak-Lojasiewicz inequality. Rearranging the terms, we get 
    \begin{align}\label{eq:n3}
        hS_n(\bg^2) &\leqslant 
        (10/3)\rho^{n}\mathbb E[f_0-f^*] 
        +  {9.37\kappa p}\,.
    \end{align}
    Note that \eqref{eq:n3} is obtained under the 
    Polyak-Lojasiewicz condition, without explicitly 
    using the strong convexity of $f$. 
    \end{proof}

\clearpage

\section{Proof of the upper bound on the error of 
parallelized RKLMC}\label{app:proof-rklmc}

In this section, we provide the 
proof of the upper bound on the error of parallelization of RKLMC.

Consider the underdamped Langevin diffusion
\begin{align}
\label{eq:uld}
\rmd \bL_t= \bV_t\,\rmd t,\qquad \text{where}
\qquad
\rmd\bV_t=- \gamma\bV_t\,\rmd t - \gamma\nabla f(\bL_t)\,\rmd t 
+\sqrt{2}\gamma\,\rmd \bW_t
\end{align}
for every $t\geqslant 0$, with  given initial
conditions $\bL_0$ and $\bV_0$. Throughout this section, 
we assume that $\bV_0\sim\mathcal N_p(\mathbf{0},\gamma\bfI_p)$ 
is independent of $\bL_0$, and the couple $(\bV_0,\bL_0)$
is independent of the Brownian motion $\bW$. We also
assume that $\bL_0$ is drawn from the target distribution
$\pi$; this implies that the process $(\bL_t,\bV_t)$ is
stationary.

In the sequel, we assume $\kappa\geqslant 10$, and we use the following shorthand notation 
\begin{align}
\delta =h/R,\qquad
    \eta = \gamma h,\qquad 
    f_n = f(\bvartheta_n),\qquad 
    \bg_n = \grad f(\bvartheta_n),\qquad 
   \bg_n^{(q,r)}=\grad f(\bvartheta_n^{(q,r)}).
\end{align}
We suppress the superscript $\sf pRKLMC$ for ease of the notation, and we rewrite these relations in the 
shorter form
\begin{align}
    \bvartheta_{n}^{(q,i)} &= \bvartheta_n + U_{ni}h\bar\alpha_{1i}\bv_n
    -\sum_{j=1}^i\int_{(j-1)\delta}^{\min(j\delta,U_{ni}h)}\big(1-e^{-\gamma(U_{ni}h-s)}\big)\,\rmd s \grad f\big(\bvartheta_n^{(q-1,j)}\big)\\
    &\qquad +U_{ni}h\sqrt{2U_{ni}\gamma \eta}\,\bar\sigma_{1i}\bar\bxi_{n,i},~~~q=1,\dots,Q-1,~i=1,\dots,R
    \label{eq:thetaU}\\
    \bvartheta_{n+1} &= \bvartheta_n + h\bar\alpha_2 \bv_n 
    -\sum_{i=1}^R\delta \big(1-e^{-\gamma h(1-U_{ni})}\big)\grad f\big(\bvartheta_n^{(Q-1,i)}\big)
    +h\sqrt{2\gamma\eta}\,\bar\sigma_2 \bxi_{n,1}
    \label{eq:thetan+1}\\
    \bv_{n+1} & = \bv_n -\eta \bar\alpha_2 \bv_n 
    - \gamma \sum_{i=1}^R\delta e^{-\gamma h(1-U_{ni})}\grad f \big(\bvartheta_n^{(Q-1,i)}\big)+ \sqrt{2\gamma\eta}
    \,\bar\sigma_3\bxi_{n,2}\label{eq:vn+1}
\end{align}
where the random vectors $\{\bar\bxi_{n,i}\}_{i=1}^{R},\bxi_{n,1},\bxi_{n,2}$ follow the standard Gaussian distribution $\mathcal{N}_p(\mathbf{0},\mathbf{I}_p)$ and are independent of $(\bvartheta_n,\bv_n)$,  $ \bar\alpha_{1i}$, %
$\bar\alpha_2$, 
and $\bar\sigma_{1i}$ are positive random variables
(with randomness inherited from $U_{ni}$ only)
satisfying
\begin{align}
\bar\alpha_{1i}\leqslant 1,\qquad
    \bar\alpha_2\leqslant 1, \qquad
    \bar\sigma_{1i}^2 \leqslant 1,\qquad
    \bar\sigma_2^2 \leqslant 1,\qquad
    \bar\sigma_3^2 \leqslant 1.
\end{align}
We define
\begin{align}
    \bar\bv_{n+1}:= \mathbb E[\bv_{n+1}],\qquad
    \bar\bvartheta_{n+1} := \mathbb E[\bvartheta_{n+1}]. 
\end{align}

The solution to SDE~\eqref{eq:uld} starting from 
$(\bv_n,\bvartheta_n)$ at the $n$-th iteration at 
time $h$ admits the following integral formulation
\begin{align}
    \bL'_t &=\bvartheta_n
    +\int_0^t\bV'_s\rmd s\\
    \bV'_t &=\bv_ne^{-\gamma t}
    -\gamma \int_0^t e^{-\gamma(t-s)}\nabla f(\bL'_s)\,\rmd s
    +\sqrt{2}\,\gamma\int_0^t e^{-\gamma(t-s)}\rmd\bW_{nh+s}\,.
    \label{eq:exactsol}
\end{align}
These expressions will be used in the proofs provided 
in the present section. %

We slightly abuse the notation here. Unlike the previous section, throughout this section, we will use the notation
\boxed{\hbox{${\bar h} = \sqrt{\gamma M}\, h$}}, which significantly 
simplifies formulas. 
We will also omit the iteration index $n$ for $U_{ni}$ and write it as $U_i$ whenever there is no risk of confusion. For simplicity of the calculations and notations, we assume throughout this section that $\arg\min_{\btheta\in\mathbb R^p} f (\btheta)=\mathbf{0}$ and $\min_{\btheta\in\mathbb R^p} f(\btheta)=0$.

\textbf{Proof overview}:
Our proof extends the approach of Theorem~2 in~\cite{yu2023langevin}, which bounds the mean discretization error and the deviation of both $\bvartheta_n$ and $\bv_n$. To adapt this analysis to our parallel computing framework, we partition the time interval $[0,h]$ into $R$ equal subintervals $[\frac{(r-1)h}{R}, \frac{rh}{R}]$, where $r = 1,\dots,R$. As shown in \Cref{lem:2}, we then compute the mean discretization error and the deviations for each of these subintervals.

\subsection{Auxiliary lemmas}

We start with some technical results required to prove Theorem~\ref{thm:rklmc}. They mainly 
assess the discretization error as well
as discounted sums of squared gradients and
velocities.

\begin{lemma}[Discretization error and deviation]%
\label{lem:2}
Let $(\bL'_t,\bV'_t)$ be the exact solution of the 
kinetic Langevin diffusion starting from $(\bvartheta_n,\bv_n)$. If $\gamma \geqslant 5M$ 
and $\gamma h
\leqslant 0.1$, it holds that
{
\begin{align}
\gamma\|\bar\bvartheta_{n+1}-\bL'_h\|_{\Ltwo} 
&\leqslant  
\frac{\gamma \bar h^4}{\sqrt{R}}   \big(1.29\delta \|\bv_n\|_{\Ltwo}
+1.29\gamma h\delta \|\bg_n\|_{\Ltwo}
+1.36\gamma\delta\sqrt{\delta p}\big)
   \\ 
    &\qquad + {\bar h}^{2Q}\gamma h 
    \Big(1.01\|\bv_n\|_{\Ltwo} + 0.51\gamma h\|\bg_n\|_{\Ltwo}
   +0.82\gamma\sqrt{hp} \Big)\\
\gamma \|\bvartheta_{n+1}-\bar\bvartheta_{n+1}\|_{\Ltwo}
    &\leqslant \gamma^2 h\delta \Big(0.22 \|\bv_n\|_{\Ltwo} +
    1.02\|\bg_n\|_{\Ltwo}
   + 0.05\gamma\sqrt{\delta p}\Big)\\
   &\qquad + (\gamma h)^{Q}\Big(0.01\|\bv_n\|_{\Ltwo}
   + 0.01 \|\bg_n\|_{\Ltwo}
   + 0.01 \gamma\sqrt{hp} \Big)\\
\|\bar \bv_{n+1} -\bV'_h\|_{\Ltwo}
&\leqslant \frac{{\bar h^2} \gamma Mh}{\sqrt{R}}  \big(1.29\delta \|\bv_n\|_{\Ltwo}
+1.29\gamma h\delta \|\bg_n\|_{\Ltwo}
+1.36\gamma\delta\sqrt{\delta p}\big)\\
   &\qquad + {\bar h}^{2Q} \Big(1.01\|\bv_n\|_{\Ltwo}
   + 0.51\gamma h\|\bg_n\|_{\Ltwo}
   +0.82\gamma\sqrt{hp} \Big)\\  
 \|\bv_{n+1}-\bar \bv_{n+1} \|_{\Ltwo}
    & \leqslant \gamma^2 h\delta  \Big(0.22\|\bv_n\|_{\Ltwo}+1.02 \|\bg_n\|_{\Ltwo} + 0.05\gamma\sqrt{\delta p}\Big)\\
   &\qquad + {(\gamma h)}^{Q}\Big(0.01\|\bv_n\|_{\Ltwo}
   + 0.01\|\bg_n\|_{\Ltwo} + 0.01\gamma\sqrt{hp} \Big)
    \,. 
\end{align}
}
\end{lemma}

\begin{proposition}\label{prop:4}
    If $\gamma \geqslant 5 M$ and $\eta \leqslant 0.1$, then, for any $n\in\mathbb N$, the iterates
    of the pRKLMC satisfy
    \begin{align}
    \eta\sum_{k=0}^{n}\varrho^{n-k}
        \|\bv_{k}\|_{\Ltwo}^2 &\leqslant   6.76\varrho^{n}\gamma \mathbb E[f_0] +  
     1.52(x_n + 1.5\sqrt{\gamma p})^2 + \frac{5.05\gamma^2 p}{m},\\
     \eta\sum_{k=0}^{n}\varrho^{n-k}
        \|\bg_{k}\|_{\Ltwo}^2 &\leqslant
        7.71\varrho^{n}\gamma \mathbb E[f_0] +  
     5.33(x_n + 1.5\sqrt{\gamma p})^2 + \frac{4.39\gamma^2 p}{m},
\end{align}
where $f_0=f(\bvartheta_0),\varrho=\exp(-mh)$ and $x_n = 
       \big(\norm{\bv_n- \bV_{nh}}^2_{\Ltwo}
       +
      \norm{\bv_n- \bV_{nh}+\gamma(  \bvartheta_n-\bL_{nh})}_{\Ltwo}^2\big)^{1/2}$.
\end{proposition}

\begin{proof}[Proof of Proposition~\ref{prop:4}]
We use the same shorthand notation as in the previous
proofs. 
Let us define $z_k = \mathbb E[\bv_k^\top \bg_k]$, and
\begin{align}
    S_n(z) &:= \sum_{k=0}^{n}\varrho^{n-k}z_k,
    &S_n(g^2) &:= \sum_{k=0}^n\varrho^{n-k} \norm{\bg_k}^2_{\Ltwo}, \\    
    S_n(f) &:= \sum_{k=0}^{n}\varrho^{n-k}
        \mathbb E[f_{k}],
    &S_n(v^2) &:= \sum_{k=0}^{n}\varrho^{n-k}
        \|\bv_{k}\|_{\Ltwo}^2 \\    
    S_n(z)_{-} &:= \max\big(0,-S_n(z)\big),
    &S_n^{+1}(v^2) &:= \sum_{k=1}^{n+1}\varrho^{n+1-k}
        \|\bv_{k}\|_{\Ltwo}^2\,.
\end{align}
We will need the following auxiliary lemma, the proof of which is postponed.
\begin{lemma}\label{lem:8}
    For any $\gamma>0$ and $h>0$ satisfying $\gamma \geqslant 5M$ and any 
    $\eta \leqslant 1/5$, the iterates of the
    randomized midpoint discretization of the kinetic 
    Langevin diffusion satisfy
    \begin{alignat}{5}%
    \|\bv_{n+1}\|_{\Ltwo}^2 
        & \leqslant & (1 - 1.66\eta ) \| \bv_n \|_{\Ltwo}^2 
        - & &2\bar\alpha_2\eta(1 - 
\eta \bar\alpha_2)  \mathbb E [\bv_n^\top \bg_{n}] +&& 2.02\eta^2 \|\bg_n\|_{\Ltwo}^2 +
        &\,3.02\gamma\eta  p
        \label{eq:help1}\\
    \mathbb E[\bv_{n+1}^\top\bg_{n+1}] 
        &\leqslant &0.91 \eta \|\bv_n\|_{\Ltwo}^2  +&&\, (1- \bar\alpha_2\eta )\mathbb E[\bv_{n}^\top\bg_{n}] -&& 0.938\eta
        \|\bg_n\|_{\Ltwo}^2  +  &\,1.51 \eta^2\gamma p
        \label{eq:help2}\\
    \gamma \mathbb E[f_{n+1} - f_n] 
        &\leqslant &0.33\eta^2 \|\bv_n\|^2_{\Ltwo} +& &\bar\alpha_2\eta \mathbb
        E[\bv_n^\top \bg_n] -&& \,0.47\eta^2\|\bg_n\|^2_{\Ltwo}  +&\, 
        0.324 \eta^3\gamma  p.\label{eq:help3}
    \end{alignat}
\end{lemma}
From the first inequality~\eqref{eq:help1} in 
\Cref{lem:8}, we infer that
\begin{align}
    S_{n}^{+1}(v^2) &\leqslant (1 - 1.66\eta ) S_n(v^2) 
    -2\bar\alpha_2\eta (1-\bar\alpha_2\eta)  S_n(z) + 2.02\eta^2 S_n(g^2) 
    + 3.02\gamma^2p/m.
\end{align}
In view of \Cref{lem:prev} and the fact that 
$\|\bv_0\|_{\Ltwo}^2 = \gamma p$, this implies that
\begin{align}
    1.66\eta S_n(v^2) + 2\bar\alpha_2\eta(1-\bar\alpha_2\eta)  S_n(z) 
     &\leqslant  S_n(v^2) - S_{n}^{+1}(v^2) + 2.02\eta^2 
     S_n(g^2) + 3.02\gamma^2p/m \\
     &\leqslant (1-\varrho) S_n(v^2) + 2.02\eta^2 
     S_n(g^2) + {3.02\gamma^2p}/{m} + \gamma p .
\end{align}
Note that $1-\varrho\leqslant \frac{m\eta}{\gamma} 
\leqslant 0.02\eta$. 
Therefore, we obtain
\begin{align}
    (1.66 - 0.02) S_n(v^2) + 2\bar\alpha_2 (1-\eta\bar\alpha_2)
    S_n(z) \leqslant 2.02\eta S_n(g^2) +  
    \frac{3.02\gamma^2 p}{m\eta},
\end{align}
that is equivalent to
\begin{align}
    \boxed{S_n(v^2)  \leqslant 1.22\bar\alpha_2(1-\bar\alpha_2\eta) S_n(z)_- 
    + 1.24\eta S_n(g^2) +  \frac{1.84\gamma^2 p}{m\eta}.}
    \label{eq:54}
\end{align}
The second step is to use the second inequality~\eqref{eq:help2} of \Cref{lem:8}.
Note that $m\eta/\gamma \leqslant 1/500$ implies $1-\varrho\geqslant0.998 m\eta/\gamma$. 
It then follows that
\begin{align}
    S_{n}^{+1} (z)
        & = (1- \bar\alpha_2\eta )
        S_n(z) - 0.938\eta  S_n(g^2) + 
        0.91 \eta S_n(v^2) + {1.51 \eta\gamma^2 p/m}.
\end{align}
This  inequality, combined with \Cref{lem:prev},  yields
\begin{align}
0.938 \eta S_n(g^2)
    &\leqslant -(\bar\alpha_2\eta +\varrho -1)S_n(z)  + 0.91\eta S_n(v^2)  + |z_{n+1}|
    + 1.51\eta\gamma^2 p/m\\
    &\leqslant \bar\alpha_2\eta S_n(z)_-  + 0.91\eta S_n(v^2)  + |z_{n+1}| +
    1.51\eta\gamma^2 p/m.
\end{align}
This can be rewritten as 
\begin{align}
    \boxed{S_n(g^2)
    \leqslant  1.07\bar\alpha_2 S_n(z)_- + 
    0.98 S_n(v^2) + \frac{1.07|z_{n+1}|}\eta
    + \frac{1.61 \gamma^2 p}{m}.}
    \label{eq:xx}
\end{align}
Let us now proceed with a similar treatment for the
last inequality of \Cref{lem:8}. Applying \Cref{lem:prev}, we get $S_{n}^{+1}(f)\geqslant 
\varrho S_n(f) - \varrho^{n+1} \mathbb E[f_0]\geqslant 
(1 - m \eta/\gamma ) S_n(f) - \varrho^{n+1} 
\mathbb E[f_0]$, which leads to
\begin{align}
    -m\eta S_n(f)
        &\leqslant\varrho^{n+1} \gamma\mathbb E[f_0] 
        + 0.33 \eta^2 S_n (v^2) + \bar\alpha_2 \eta  
        S_n(z) - 0.474\eta^2 S_n(g^2) + 0.324
        \frac{\eta^2\gamma^2 p}{m}.
\end{align}
From this inequality, and the Polyak-Lojasievicz 
condition, one can infer that
\begin{align}\label{eq:Snz2}
    \boxed{
    \bar\alpha_2 S_n(z)_-
        \leqslant \varrho^{n+1} \gamma \mathbb E
        [f_0]/\eta + 0.33 \eta S_n(v^2) + (0.5 - 
        0.474\eta) S_n(g^2) + 0.324 \frac{\eta\gamma^2 
        p}{m}.}
\end{align}
Combining \eqref{eq:Snz2} with \eqref{eq:54},
we get
\begin{align}
    S_n(v^2) \leqslant 1.22 \Big(\varrho^{n+1} 
    \gamma\mathbb E[f_0]/\eta &+ 0.33 \eta S_n(v^2) + 
    (0.5 - 0.474\eta) S_n(g^2) + 0.324 \frac{\eta\gamma^2 p}{m}\Big)\\
    & + 1.24\eta S_n(g^2)  +\frac{1.84\gamma^2 p
    }{m\eta}.
\end{align}
Since $\eta\leqslant 0.1$, it follows then
\begin{align}\label{eq:h1}   
     \boxed{S_n(v^2) \leqslant 0.71\Big( S_n(g^2) + \frac{1.8 \varrho^{n}\gamma}{
     \eta} \,\mathbb E[f_0] +  
     \frac{2.71\gamma^2p}{m\eta}\Big)\,.}
\end{align}
Similarly, combining \eqref{eq:Snz2} and
\eqref{eq:xx}, we get
\begin{align}
    S_n(g^2) \leqslant 1.07 \Big(\varrho^{n}\gamma 
    \mathbb E[f_0]/\eta &+ 0.33 \eta S_n(v^2) + 
    (0.5 - 0.474\eta) S_n(g^2) + 0.324 \frac{\eta\gamma^2 p}{m}\Big)\\
    & + 0.98 S_n(v^2) + \frac{1.07|z_{n+1}|}{\eta} + \frac{1.61\gamma^2 p
    }{m}.
\end{align}
Since $\eta\leqslant 0.1$, it follows then
\begin{align}\label{eq:h2}
    \boxed{
    S_n(g^2) \leqslant 0.8  S_n(v^2) + \frac{ 2.302\varrho^{n}\gamma}{\eta}\, 
    \mathbb E[f_0] +  
     \frac{2.302|z_{n+1}|}{\eta} + \frac{3.537\gamma^2 p}{m}.
    }
\end{align}
Equations \eqref{eq:h1} and \eqref{eq:h2} 
together yield 
\begin{align}
    S_n(g^2) &\leqslant 0.568 \Big( 
    S_n(g^2) + \frac{1.8\varrho^{n}\gamma}{
     \eta} \,\mathbb E[f_0] +  
     \frac{2.71\gamma^2p}{m\eta}\Big) 
     + \frac{ 2.302\varrho^{n}\gamma}{\eta}\, 
    \mathbb E[f_0] +  
     \frac{2.302|z_{n+1}|}{\eta} + \frac{3.537\gamma^2 p}{m}\\
     &\leqslant 0.568
    S_n(g^2) 
     + \frac{ 3.33\varrho^{n}\gamma}{\eta}\, 
    \mathbb E[f_0] +  
     \frac{2.302|z_{n+1}|}{\eta} + \frac{1.893\gamma^2 p}{m\eta}.
\end{align}
Hence, we get
\begin{align}
    \boxed{
    S_n(g^2) \leqslant \frac{ 7.71\varrho^{n}\gamma}{\eta}\, \mathbb E[f_0] +  
     \frac{5.33|z_{n+1}|}{\eta} + \frac{4.39\gamma^2 p}{m\eta}.}
\end{align}
Using once again equation \eqref{eq:h1}, 
we arrive at 
\begin{align}
    S_n(v^2) &\leqslant 0.71\Big( \frac{ 7.71\varrho^{n}\gamma}{\eta}\, \mathbb E[f_0] +  
     \frac{5.33|z_{n+1}|}{\eta} + \frac{4.39\gamma^2 p}{m\eta} + \frac{1.8 \varrho^{n+1}}{
     \eta} \,\mathbb E[f_0] +  
     \frac{2.71\gamma^2p}{m\eta}\Big),
\end{align}
which is equivalent to
\begin{align}
    \boxed{
    S_n(v^2) \leqslant  \frac{ 6.76\varrho^{n}\gamma}{\eta}\, \mathbb E[f_0] +  
     \frac{3.79|z_{n+1}|}{\eta} + \frac{5.05\gamma^2 p}{m\eta}.}
\end{align}
To complete the proof of the proposition, 
it remains to establish a suitable upper
bound on $|z_{n+1}|$. To this end, we note
that 
\begin{align}
\norm{\bg_n}_{\Ltwo} &\leqslant 
    M \norm{\bvartheta_n-\bL_{nh}}_{\Ltwo}+\sqrt{M p}\\
    &\leqslant 0.2\norm{\gamma(\bvartheta_n-\bL_{nh})}_{\Ltwo} + \sqrt{0.2\gamma p}\\
    &\leqslant 0.3 (x_n + 1.5\sqrt{\gamma p})\\
\norm{\bv_n}_{\Ltwo}
    &\leqslant \norm{\bv_n-\bV_{nh}}_{\Ltwo}+\sqrt{\gamma p}\\
    &\leqslant \norm{\bv_n-\bV_{nh}}_{\Ltwo}+\sqrt{\gamma p}\\
    &\leqslant x_n +\sqrt{\gamma p}.
\end{align}
Then,  following the same steps as those
used in the proof of the second inequality of \Cref{lem:8}, one can infer that 
\begin{align}
    |z_{n+1}| 
    &\leqslant  
        |z_n| + 0.97\eta\|
        \bg_n\|_{\Ltwo}^2 +  0.51 \eta\|\bv_n\|_{
        \Ltwo}^2 +  0.09  \eta^2\gamma p\\
    &\leqslant  \|
        \bg_n\|_{\Ltwo}\|
        \bv_n\|_{\Ltwo}
         + 0.1\|\bg_n\|_{\Ltwo}^2 +  0.051 \|\bv_n\|_{
        \Ltwo}^2 +  0.001 \gamma p\\
    &\leqslant 1.1\|\bg_n\|^2_{\Ltwo} + 
    0.301\|\bv_n\|^2_{\Ltwo} + 0.001\gamma p\\
    &\leqslant 0.099(x_n + 1.5\sqrt{\gamma p})^2 
        + 0.301(x_n + 1.1\sqrt{p})^2\\
    &\leqslant 0.4\big(x_n +1.5 \sqrt{p}\big)^2.
\label{eq:help4}
\end{align}
This completes the proof of the proposition. 
\end{proof}

\subsection{Proof of Theorem~\ref{thm:rklmc}}

Let $\bvartheta_{n+U},\bvartheta_{n+1},\bv_{n+1}$ 
be the iterates of Algorithm.
Let $(\bL_t,\bV_t)$ be the kinetic Langevin
diffusion, coupled with $(\bvartheta_n,\bv_n)$ through the same Brownian motion and starting from a random point $(\bL_0,\bV_0)\propto \exp(-f (\btheta) - \frac{1}{2}\|\bv\|^2)$ such that $\bV_0 = \bv_0$. Let $(\bL'_t,\bV'_t)$ be the kinetic 
Langevin diffusion defined on $[0,h]$ using the same Brownian motion and starting from $(\bvartheta_n, 
\bv_n)$. 

Our goal will be to bound the term $x_n$ defined by
\begin{align}
    x_n = \bigg\| \bfC
    \begin{bmatrix}
        \bv_n- \bV_{nh}\\
        \bvartheta_n-\bL_{nh}
    \end{bmatrix}
    \bigg\|_{\Ltwo}
    \quad \text{with}\quad 
    \bfC = \begin{bmatrix}
    \mathbf I_p & \mathbf 0_{p}\\
    \mathbf I_p & \gamma\mathbf I_p
    \end{bmatrix}.
\end{align}
To this end, define
\begin{align}
    \bar\bv_{n+1} = \mathbb E_{U}[\bv_{n+1}],\qquad
    \bar\bvartheta_{n+1} = \mathbb E_{U}[\bvartheta_{n+1}]. 
\end{align}
Since $(\bV_{(n+1)h},\bL_{(n+1)h})$ are independent
of $U$, we have 
\begin{align}
    x_{n+1}^2 & = \bigg\| \bfC
    \begin{bmatrix}
        \bv_{n+1}- \bar\bv_{n+1}\\
        \bvartheta_{n+1}-\bar\bvartheta_{n+1}
    \end{bmatrix}
    \bigg\|_{\Ltwo}^2 + 
    \bigg\| \bfC
    \begin{bmatrix}
        \bar\bv_{n+1}- \bV_{(n+1)h}\\
        \bar\bvartheta_{n+1}-\bL_{(n+1)h}
    \end{bmatrix}
    \bigg\|_{\Ltwo}^2.
\end{align}
Using the triangle inequality and Proposition 1 in~\cite{yu2023langevin} (See also Proposition 1 from~\citep{dalalyan_riou_2018}), we get
\begin{align}
    \bigg\| \bfC
    \begin{bmatrix}
        \bar\bv_{n+1}- \bV_{(n+1)h}\\
        \bar\bvartheta_{n+1}-\bL_{(n+1)h}
    \end{bmatrix}
    \bigg\|_{\Ltwo}&\leqslant 
    \bigg\| \bfC
    \begin{bmatrix}
        \bar\bv_{n+1}- \bV'_{h}\\
        \bar\bvartheta_{n+1}-\bL'_{h}
    \end{bmatrix}
    \bigg\|_{\Ltwo}   
    +
    \bigg\| \bfC
    \begin{bmatrix}
        \bar\bV'_{h}- \bV_{(n+1)h}\\
        \bar\bL'_{h}-\bL_{(n+1)h}
    \end{bmatrix}
    \bigg\|_{\Ltwo}\\
    &\leqslant \bigg\| \bfC
    \begin{bmatrix}
        \bar\bv_{n+1}- \bV'_{h}\\
        \bar\bvartheta_{n+1}-\bL'_{h}
    \end{bmatrix}
    \bigg\|_{\Ltwo}   
    + \varrho x_{n}
\end{align}
where $\varrho = e^{-mh}$. Combining these
inequalities, we get
\begin{align}
    x_{n+1}^2 &\leqslant \big(\varrho x_{n} + y_{n+1}\big)^2 + z_{n+1}^2 
\end{align}
where
\begin{align}
    y_{n+1} = \bigg\| \bfC
    \begin{bmatrix}
        \bar\bv_{n+1} - \bV'_{h}\\
        \bar\bvartheta_{n+1} - \bL'_{h}
    \end{bmatrix}
    \bigg\|_{\Ltwo},\qquad 
    z_{n+1} = \bigg\| \bfC
    \begin{bmatrix}
        \bv_{n+1}- \bar\bv_{n+1}\\
        \bvartheta_{n+1}-\bar\bvartheta_{n+1}
    \end{bmatrix}
    \bigg\|_{\Ltwo}.
\end{align}
This yields\footnote{One can check by induction, 
that if for some sequences $x_n,y_n,z_n$ and 
some $\varrho \in(0,1)$ it holds that $x_{n+1}^2
\leqslant  (\varrho x_{n} + y_{n+1})^2 + z_{n+1
}^2$, then necessarily $x_n\leqslant \varrho^n 
x_0 + \sum_{k=1}^n \varrho^{n-k}y_k + (\sum_{
k=1}^n \varrho^{2(n-k)}z_k^2)^{1/2}$ for every 
$n\in\mathbb N$.}
\begin{align}
    x_n &\leqslant \varrho^n x_0 + \sum_{k=1}^n 
    \varrho^{n-k} y_k + \bigg(\sum_{k=1}^n
    \varrho^{2(n-k)} z_k^2\bigg)^{1/2}\,.
    \label{eq:xnC}
\end{align}
Using the fact that $\|\bfC[a, b]^\top\|^2 = 
\|a\|^2 + \|a + \gamma b\|^2 \leqslant 3\|a\|^2 
+ 2\gamma^2\|b\|^2$, we arrive at
\begin{align}
    y_{n+1}^2&\leqslant 3\|\bar\bv_{n+1} - \bV'_h 
    \|_{\Ltwo}^2 + 2\gamma^2 \|\bar\bvartheta_{
    n+1} - \bL'_h\|_{\Ltwo}^2, \label{eq:yn}\\
    z_{n+1}^2&\leqslant 3\|\bv_{n+1} - \bar\bv_{
    n+1} \|_{\Ltwo}^2 + 2\gamma^2 \|\bvartheta_{
    n+1} - \bar\bvartheta_{n+1}\|_{\Ltwo}^2.
    \label{eq:zn}
\end{align}
Note that 
\begin{align}
    \bigg\{\sum_{k=1}^n \varrho^{n-k} y_k\bigg\}^2     
   &\leqslant \sum_{k=1}^n \frac{\varrho^{n-k}y_k^2}{{1-\varrho}}   
   \leqslant \frac{1.001\gamma}{m\eta}\sum_{k=1}^n 
   \varrho^{n-k} (3\|\bar\bv_{k} - \bV'_h 
    \|_{\Ltwo}^2 + 2\gamma^2 \|\bar\bvartheta_{k} 
    - \bL'_h\|_{\Ltwo}^2) .
    \label{eq:xn_rkl}
\end{align}
We then have
\begin{align}
 x_n &\leqslant \varrho^n x_0 + \bigg(\frac{1.001\gamma}{m\eta}\sum_{k=1}^n 
   \varrho^{n-k} (3\|\bar\bv_{k} - \bV'_h 
    \|_{\Ltwo}^2 + 2\gamma^2 \|\bar\bvartheta_{k} 
    - \bL'_h\|_{\Ltwo}^2) \bigg)^{1/2} \\
   &\qquad  + \bigg(\sum_{k=1}^n
    \varrho^{2(n-k)} 3\|\bv_{k} - \bar\bv_{
    k} \|_{\Ltwo}^2 + 2\gamma^2 \|\bvartheta_{
    k} - \bar\bvartheta_{k}\|_{\Ltwo}^2 \bigg)^{1/2}\,.
\end{align}
By \Cref{lem:2}, we find
\begin{align}
\gamma^2\|\bar\bvartheta_{n+1}-\bL'_h\|_{\Ltwo}^2 
&\leqslant  
0.002\Bigg[{\frac{M\gamma^7h^6}{R}} (\delta^2 \|\bv_n\|^2_{\Ltwo}+{\gamma^2\delta^2h^2}\|\bg_n\|_{\Ltwo}^2+\gamma^2\delta^3p)\\
&\qquad \qquad + M\gamma^3h^2(\gamma h)^{4Q-4}(h^2\|\bv_n\|^2_{\Ltwo}+\gamma^2h^4\|\bg_n\|_{\Ltwo}^2+\gamma^2h^3p)\Bigg]
\\
 \|\bar \bv_{n+1} -\bV'_h\|_{\Ltwo}^2
    & \leqslant 
0.13\Bigg[{\frac{M\gamma^7h^6}{R}} (\delta^2 \|\bv_n\|^2_{\Ltwo}+{\gamma^2\delta^2h^2}\|\bg_n\|_{\Ltwo}^2+\gamma^2\delta^3p)\\
&\qquad \qquad + M\gamma^3h^2(\gamma h)^{4Q-4}(h^2\|\bv_n\|^2_{\Ltwo}+\gamma^2h^4\|\bg_n\|_{\Ltwo}^2+\gamma^2h^3p)\Bigg]
    \\
\gamma^2 \|\bvartheta_{n+1}-\bar\bvartheta_{n+1}\|_{\Ltwo}^2 
&\leqslant  
3.122\Big((\gamma h)^{2Q}+\frac{\gamma^4h^4}{R^2} \Big)  (\|\bv_n\|_{\Ltwo}^2+ \|\bg_n\|_{\Ltwo}^2
   + \gamma^2 ph)\\
\| \bv_{n+1}-\bar \bv_{n+1} \|_{\Ltwo}^2
& \leqslant 3.122\Big((\gamma h)^{2Q}+\frac{\gamma^4h^4}{R^2} \Big)  (\|\bv_n\|_{\Ltwo}^2+ \|\bg_n\|_{\Ltwo}^2
   + \gamma^2 ph)
    \,. 
\end{align}
Therefore, we infer from~\eqref{eq:xn_rkl} that
\begin{align}
x_n  &\leqslant    \varrho^n x_0  +0.63\Big(\frac{\gamma}{m\eta}\sum_{k=1}^n \varrho^{n-k}\bigg[{\frac{M\gamma^7h^6\delta^2}{R}}(\|\bv_n\|_{\Ltwo}^2+{\gamma^2h^2}\|\bg_n\|_{\Ltwo}^2+\gamma^2\delta p)\\
&\qquad\qquad \qquad \qquad\qquad +M\gamma^3h^4(\gamma h)^{4Q-4}(\|\bv_n\|_{\Ltwo}^2+\gamma^2h^2\|\bg_n\|_{\Ltwo}^2+\gamma^2ph)\Big]\bigg)^{1/2}\\
&\qquad \qquad + 3.96\bigg(\sum_{k=1}^n\varrho^{2(n-k)}\Big((\gamma h)^{2Q}+\frac{\gamma^4h^4}{R^2}\Big)(\|\bv_n\|_{\Ltwo}^2+\|\bg_n\|_{\Ltwo}^2+\gamma^2 hp) \bigg)^{1/2}\\
&\leqslant
 \varrho^n x_0  
 +\Big[0.63\big({\kappa\gamma^6h^4\delta^2/R}+\kappa(\gamma h)^{4Q-2}\big)^{1/2}+ 3.96\big(\gamma^3h\delta^2+(\gamma h)^{2Q-1}\big)^{1/2}\Big]\Big(\gamma h\sum_{k=1}^n \varrho^{n-k}\|\bv_n\|_{\Ltwo}^2\Big)^{1/2}\\
 &\quad + \Big[0.007\big({\kappa\gamma^6h^4\delta^2/R}+\kappa(\gamma h)^{4Q-2}\big)^{1/2}+3.96\big(\gamma^3 h\delta^2+(\gamma h)^{2Q-1}\big)^{1/2}\Big] \Big(\gamma h\sum_{k=1}^n \varrho^{n-k}\|\bg_n\|_{\Ltwo}^2\Big)^{1/2}\\
&\quad +
\Big[0.07\big({\kappa\gamma^6h^4\delta^2/R}+\kappa(\gamma h)^{4Q-2}\big)^{1/2}+0.4\big(\gamma^3 h\delta^2+(\gamma h)^{2Q-1}\big)^{1/2}\Big]\gamma\sqrt{p/m}\,.
\end{align}
By Proposition~\ref{prop:4}, we then obtain
\begin{align}
x_n
&\leqslant  \varrho^n x_0 
+\Big[0.793\big({\kappa\gamma^6h^4\delta^2/R}+\kappa(\gamma h)^{4Q-2}\big)^{1/2}
+14.03\big(\gamma^3h\delta^2+(\gamma h)^{2Q-1}\big)^{1/2}\Big]x_n\\
&\qquad \quad+ \Big[1.66\big({\kappa\gamma^6h^4\delta^2/R}+\kappa(\gamma h)^{4Q-2}\big)^{1/2}
+21.3\big(\gamma^3h\delta^2+(\gamma h)^{2Q-1}\big)^{1/2}\Big]\sqrt{\gamma\varrho^n\E[f_0]}\\
&\qquad \quad+ 
\Big[1.19\big({\kappa\gamma^6h^4\delta^2/R}+\kappa(\gamma h)^{4Q-2}\big)^{1/2}
+21.04\big(\gamma^3h\delta^2+(\gamma h)^{2Q-1}\big)^{1/2}\Big]\sqrt{\gamma p}\\
&\qquad \quad+ 
\Big[1.44\big({\kappa\gamma^6h^4\delta^2/R}+\kappa(\gamma h)^{4Q-2}\big)^{1/2}
+17.2\big(\gamma^3h\delta^2+(\gamma h)^{2Q-1}\big)^{1/2}\Big]\gamma\sqrt{p/m}\,.
\end{align}
Note that $\kappa =\frac{M}{m}\leqslant \frac{\gamma}{5m}, Q\geqslant 2,$ the third term on the right-hand side can be bounded as follows.
\begin{align}
\Big[1.19&\big({\kappa\gamma^6h^4\delta^2/R}+\kappa(\gamma h)^{4Q-2}\big)^{1/2}
+21.04\big(\gamma^3h\delta^2+(\gamma h)^{2Q-1}\big)^{1/2}\Big]\sqrt{\gamma p}\\
&\leqslant
1.19\big(\gamma^6h^4\delta^2+(\gamma h)^{4Q-2}\big)^{1/2}\gamma \sqrt{p/5m}
+ 21.04 \big(\gamma^3h\delta^2+(\gamma h)^{2Q-1}\big)^{1/2} \gamma \sqrt{p/5m}\\
&\leqslant 1.19 (0.1)^{3/2} \big(\gamma^3 h\delta^2+ (\gamma h)^{2Q-1}\big)^{1/2}\gamma \sqrt{p/5m}
+ 21.04\big(\gamma^3h\delta^2+(\gamma h)^{2Q-1}\big)^{1/2} \gamma \sqrt{p/5m}\\
&\leqslant 9.427 \big(\gamma^3 h\delta^2+ (\gamma h)^{2Q-1}\big)^{1/2}\gamma \sqrt{p/m}
\,.
\end{align}
Combining this with the previous display and simplifying the expression then gives
\begin{align}
x_n
&\leqslant  \varrho^n x_0 
+\Big[0.793\big({\kappa\gamma^6h^4\delta^2/R}+\kappa(\gamma h)^{4Q-2}\big)^{1/2}
+14.03\big(\gamma^3h\delta^2+(\gamma h)^{2Q-1}\big)^{1/2}\Big]x_n\\
&\qquad \quad+ \Big[1.66{\kappa\gamma^6h^4\delta^2/R}+\kappa(\gamma h)^{4Q-2}\big)^{1/2}
+21.3\big(\gamma^3h\delta^2+(\gamma h)^{2Q-1}\big)^{1/2}\Big]\sqrt{\gamma\varrho^n\E[f_0]}\\
&\qquad \quad+ 
\Big[1.44{\kappa\gamma^6h^4\delta^2/R}+\kappa(\gamma h)^{4Q-2}\big)^{1/2}
+26.63\big(\gamma^3h\delta^2+(\gamma h)^{2Q-1}\big)^{1/2}\Big]\gamma\sqrt{p/m}\,.
\end{align}
When ${\kappa\gamma^6h^4\delta^2/R}+\kappa(\gamma h)^{4Q-2}\leqslant 0.1^6,$ note that $\delta = h/R$ and $R\geqslant 1,$ we have
\begin{align}
\big(\gamma^3h\delta^2+(\gamma h)^{2Q-1}\big)^2\leqslant 2\gamma^6 h^2\delta^4+2(\gamma h)^{4Q-2}\leqslant {2\gamma^6 h^4\delta^2/R}+2(\gamma h)^{4Q-2}\leqslant 2\times 0.1^6\,.
\end{align}
This implies
\begin{align}
x_n
&\leqslant  \varrho^n x_0 
+0.529 x_n
+ 0.803 \sqrt{\gamma\varrho^n\E[f_0]}\\
&\quad +\Big[1.44\big({\kappa\gamma^6h^4\delta^2/R}+\kappa(\gamma h)^{4Q-2}\big)^{1/2}
+26.63\big(\gamma^3h\delta^2+(\gamma h)^{2Q-1}\big)^{1/2}\Big]\gamma\sqrt{p/m}\,.
\end{align}
Using the fact that $x_0 = \gamma 
\wass_2(\nu_0,\pi)$ and $x_n\geqslant \gamma 
\wass_2(\nu_n,\pi)/\sqrt{2}$, we get
\begin{align}
\wass_2(\nu_n,\pi)
&\leqslant 3.04\varrho^n\wass_2(\nu_0,\pi)
+2.42\sqrt{\frac{\varrho^n}{\gamma}\E[f_0] }
+ 4.33 \big({\gamma^6h^4\delta^2/R}+(\gamma h)^{4Q-2}\big)^{1/2}\sqrt{\kappa p/m}\\
&\qquad + 80.11 \big(\gamma^3h\delta^2+(\gamma h)^{2Q-1}\big)^{1/2}\sqrt{p/m} \\
&\leqslant 3.04\varrho^n\wass_2(\nu_0,\pi)
+1.09\sqrt{\frac{\varrho^n}{m}\E[f_0] }
+ 4.33 \big({\gamma^6h^4\delta^2/R}+(\gamma h)^{4Q-2}\big)^{1/2}\sqrt{\kappa  p/m} \\
&\qquad + 80.11 \big(\gamma^3h\delta^2+(\gamma h)^{2Q-1}\big)^{1/2}\sqrt{p/m}
\end{align}
as desired.
The last step follows from the fact that $\gamma \geqslant 5M \geqslant 5m$.

\subsection{Proofs of the technical lemmas}
We now provide proof of the technical lemmas that we used in this section.
To this end, we first give three auxiliary Lemmas~\ref{lem:9}-\ref{lem:dev1}, which quantify the precision of the midpoint $\bvartheta_n^{(Q-1,i)}.$

\begin{lemma}%
\label{lem:9}
For every  $h>0$ and $Q\geqslant 2$, it holds that

{
\begin{align}
\frac{1}{R}\sum_{i=1}^R  \|\bvartheta_n^{(Q-1,i)} -  \bL'_{U_i h}\|_{\Ltwo}
    &\leqslant 
 \frac{\bar h^2}{\sqrt{R}}   \Big(\sum_{i=0}^{Q-2} \bar h^{2i}\Big)\big(1.28\delta \|\bv_n\|_{\Ltwo}
+1.28\gamma h\delta \|\bg_n\|_{\Ltwo}
+1.35\gamma\delta\sqrt{p\delta}\big)
   \\ 
    &\qquad + {\bar h}^{2(Q-2)}e^{0.5{\bar h^2}} 
    \Big(h\|\bv_n\|_{\Ltwo} + 0.5\gamma h^2\|\bg_n\|_{\Ltwo}
   +\sqrt{(2/3)ph}\gamma h\Big)
\end{align}
}

\end{lemma}

\begin{proof}[Proof of \Cref{lem:9}]

{
For any $s,u\in[0,h]$ such that $s<u$, it holds that
\begin{align}
\|\bL'_s-\bL'_u\|_{\Ltwo}
&\leqslant \Big|\frac{e^{-\gamma s}-e^{-\gamma u}}{\gamma}\Big|\|\bv_n\|_{\Ltwo}
+ \Big| \frac{e^{-\gamma u}+\gamma u-e^{-\gamma s}-\gamma s}{\gamma}\Big|\|\bg_n\|_{\Ltwo}
+ \gamma \sqrt{\frac{2p(u-s)^3}{3}} \\
&\qquad +\gamma\bigg\| \int_0^s \frac{1-e^{-\gamma(s-t)}}{\gamma} (\grad f(\bL'_t)-\bg_n)dt-\int_0^u  \frac{1-e^{-\gamma(u-t)}}{\gamma} (\grad f(\bL'_t)-\bg_n)dt\bigg\|_{\Ltwo}\,.
\end{align}
Note that $e^{-\gamma s}-e^{\gamma u}\leqslant (u-s)$ and by Mean Value Theorem, it holds for some $\zeta\in[s,u]$ that
\begin{align}
\Big| \frac{e^{-\gamma u}+\gamma u-e^{-\gamma s}-\gamma s}{\gamma}\Big|
= |(1-e^{-\gamma \zeta})(u-s)|
 \leqslant \gamma h(u-s)\,.
\end{align}
Moreover, it holds that
\begin{align}
&\bigg\| \int_0^s \frac{1-e^{-\gamma(s-t)}}{\gamma} (\grad f(\bL'_t)-\bg_n)dt-\int_0^u  \frac{1-e^{-\gamma(u-t)}}{\gamma} (\grad f(\bL'_t)-\bg_n)dt\bigg\|_{\Ltwo}\\
&=\bigg\| \int_0^s \frac{e^{-\gamma(u-t)}-e^{-\gamma(s-t)}}{\gamma} (\grad f(\bL'_t)-\bg_n)dt-\int_s^u  \frac{1-e^{-\gamma(u-t)}}{\gamma} (\grad f(\bL'_t)-\bg_n)dt\bigg\|_{\Ltwo} \\
&\leqslant (u-s)\int_0^u \|\grad f(\bL'_t)-\bg_n\|_{\Ltwo} dt\\
&\leqslant (u-s)M\int_0^u \| \bL'_t-\bL'_0 \|_{\Ltwo}dt\,.
\end{align}
The inequality (42) in~\cite{yu2023langevin}  
states that 
\begin{equation}
\label{eq:prev2}
\|\bL'_s-\bL'_0\|_{\Ltwo}
    \leqslant e^{0.5\bar h^2}\Big(s\|\bv_n\|_{\Ltwo}
    + 0.5\gamma s^2\|\bg_n\|_{\Ltwo}
    + \sqrt{(2/3)ps}\,\gamma s \Big)\,,\qquad s\in[0,h]\,.
\end{equation}
Plugging this into the previous display yields
\begin{align}
&\bigg\| \int_0^s \frac{1-e^{-\gamma(s-t)}}{\gamma} (\grad f(\bL'_t)-\bg_n)dt-\int_0^u  \frac{1-e^{-\gamma(u-t)}}{\gamma} (\grad f(\bL'_t)-\bg_n)dt\bigg\|_{\Ltwo}\\
&\leqslant (u-s)M e^{0.5\bar h^2}\int_0^u\Big(t\|\bv_n\|_{\Ltwo}
    + 0.5\gamma t^2\|\bg_n\|_{\Ltwo}
    + \sqrt{(2/3)pt}\,\gamma t \Big) dt\\
&\leqslant (u-s)Me^{0.5\bar h^2}\Big(\frac{u^2}{2}\|\bv_n\|_{\Ltwo}+\frac{\gamma u^3}{6}\|\bg_n\|_{\Ltwo}
+ \sqrt{\frac{8p}{75}}\gamma u^{5/2}\Big)\,.
\end{align}
Collecting pieces then provides us with
\begin{align}
 \|\bL'_s-\bL'_u\|_{\Ltwo}
&\leqslant (u-s)  \|\bv_n\|_{\Ltwo}+\gamma h(u-s)  \|\bg_n\|_{\Ltwo}
    +\gamma \sqrt{\frac{2p(u-s)^3}{3}} \\
   &\quad  +(u-s)M\gamma e^{0.5\bar h^2}\Big(\frac{u^2}{2}\|\bv_n\|_{\Ltwo}+\frac{\gamma u^3}{6}\|\bg_n\|_{\Ltwo}
+ \sqrt{\frac{8p}{75}}\gamma u^{5/2}\Big)\,.
\label{eq:disL}
\end{align}
By the definition of $\bvartheta_n^{(q,i)}$, setting temporarily 
$\triangle_i(s) = 1-e^{-\gamma(U_i h -s)}\in [0, \gamma h]$ for every
$s\in[0,U_i h]$, we have
\begin{align}
    \|\bvartheta_n^{(q,i)} -  \bL'_{U_i h}\|
    &\leqslant    \bigg\|\sum_{j=1}^i\int_{(j-1)\delta}^{j\delta
    \wedge U_ih}\triangle_i(s) \,\rmd s \grad f\big(\bvartheta_n^{(q-1,j)} \big)  -\int_0^{U_ih} \triangle_i(s)\grad f(\bL'_s)\,\rmd s
    \bigg\|\\
    &\leqslant    \bigg\|\sum_{j=1}^i\int_{(j-1)\delta}^{j\delta
    \wedge U_ih}\triangle_i(s) \big(\grad f(\bvartheta_n^{(q-1,j)}) 
    - \grad f(\bL'_s)\big)\,\rmd s \bigg\|\\
    & \leqslant M\gamma h\delta\sum_{j=1}^i \|\bvartheta_n^{(q-1,j)}-\bL'_{U_jh}\|
     + \Bigg\|\sum_{j=1}^{i-1} \int_{(j-1)\delta}^{j\delta}\triangle_i(s) \big(\grad f(\bL'_s)-\grad f((\bL'_{U_jh})\big)\rmd s\Bigg\|\\
     &\qquad +  M\gamma h \int_{(i-1)\delta}^{U_i h} \|\bL'_{U_j h}-\bL'_s\|\rmd s \,,
     \label{eq:dis1}
\end{align}
where in the last inequality we have used the Lipschitz property of 
$\nabla f$ and the triangle inequality. 
Note that $Z_j:=\int_{(j-1)\delta}^{j\delta}\triangle_i(s) \big(\grad f(\bL'_s)-\grad f((\bL'_{U_jh})\big)\rmd s$ are independent random variables with zero mean.
In view of inequality~\eqref{eq:disL}, we then obtain the following result
\begin{align}
 \Bigg\|\sum_{j=1}^{i-1} Z_j\Bigg\|_{\Ltwo}^2
 &= \sum_{j=1}^{i-1} \Bigg\| \int_{(j-1)\delta}^{j\delta}\triangle_i(s) \big(\grad f(\bL'_s)-\grad f((\bL'_{U_jh})\big)\rmd s\Bigg\|_{\Ltwo}^2\\
 &\leqslant \gamma^2 h^2 M^2\delta \sum_{j=1}^{i-1}\int_{(j-1)\delta}^{j\delta}\|\bL'_s-\bL'_{U_jh}\|_{\Ltwo}^2 \rmd s\\
 &\leqslant  \gamma^2 h^2 M^2\delta \sum_{j=1}^{i-1}\int_{(j-1)\delta}^{j\delta} \Big\{ 1.53\delta^2 \|\bv_n\|_{\Ltwo}^2
 + 1.53 \gamma^2h^2\delta^2 \|\bg_n\|_{\Ltwo}^2
 + 1.01\gamma^2p\delta^3 \Big\} \rmd s\\
 &= \gamma^2 h^2 M^2\delta \sum_{j=1}^{i-1}\Big (1.53\delta^3\|\bv_n\|_{\Ltwo}^2
 + 1.53 \gamma^2h^2\delta^3\|\bg_n\|_{\Ltwo}^2
 + 1.01\gamma^2p\delta^4\Big) \,.
\end{align}
The first inequality follow from Cauchy–Schwarz inequality. 
Therefore, 
\begin{align}
\Bigg\|\sum_{j=1}^{i-1} \int_{(j-1)\delta}^{j\delta}\triangle_i(s) \big(\grad f(\bL'_s)&-\grad f((\bL'_{U_jh})\big)\rmd s\Bigg\|_{\Ltwo}
=  \Bigg\|\sum_{j=1}^{i-1} Z_j\Bigg\|_{\Ltwo}\\
&\leqslant  M\gamma h \delta \sqrt{i-1} \big(1.24\delta \|\bv_n\|_{\Ltwo}
+1.24\gamma h\delta \|\bg_n\|_{\Ltwo}
+1.01\gamma\delta\sqrt{p\delta}\big)\,.
 \label{eq:dis2}
\end{align}
Similarly, in view of~\eqref{eq:disL}, it holds that
\begin{align}
\Big\| \int_{(i-1)\delta}^{U_ih} \|\bL'_s&-\bL'_{U_ih}\| \rmd s\Big\|_{\Ltwo}^2
\leqslant \frac{1}{\delta}\int_0^\delta \Big(\int_0^u \|\bL'_{(i-1)\delta +u}-\bL'_{(i-1)\delta+v}\|_{\Ltwo}^2\rmd v\Big)^2\rmd u\\
&\leqslant \frac{1}{\delta}\int_0^\delta \Big(\int_0^u \big\{1.01|u-v|\|\bv_n\|_{\Ltwo}+1.01\gamma h|u-v|\|\bg_n\|_{\Ltwo}
+0.82\gamma\sqrt{p\delta}|u-v|\big\}\rmd v\Big)^2\rmd u\\
&\leqslant  \frac{1}{\delta}\int_0^\delta \big(0.51u^2\|\bv_n\|_{\Ltwo}+0.51\gamma hu^2\|\bg_n\|_{\Ltwo} 
+0.41\gamma u^2\sqrt{p\delta}\big)^2\rmd u\\
&\leqslant \delta^4 \big(0.16\|\bv_n\|_{\Ltwo}^2
+0.16\gamma^2h^2\|\bg_n\|_{\Ltwo}^2
+0.11\gamma^2p\delta\big)\,.
\end{align}
Combining~\eqref{eq:dis1}, \eqref{eq:dis2} and the last display, we obtain
\begin{align}
\|\bvartheta_n^{(q,i)} -  \bL'_{U_i h}\|
&\leqslant  \frac{\bar h^2}{R}\sum_{j=1}^R\|\bvartheta_n^{(q-1,j)}-\bL'_{U_jh}\|_{\Ltwo} 
+  M\gamma h\delta \big(0.4\delta\|\bv_n\|_{\Ltwo}+0.4\delta \|\bg_n\|_{\Ltwo} 
+ 0.34\gamma \delta \sqrt{p\delta}\big)\\
&\qquad +M\gamma h \delta \sqrt{R} \big(1.24\delta \|\bv_n\|_{\Ltwo}
+1.24\gamma h\delta \|\bg_n\|_{\Ltwo}
+1.01\gamma\delta\sqrt{p\delta}\big)\\
&\leqslant  \frac{\bar h^2}{R}\sum_{j=1}^R\|\bvartheta_n^{(q-1,j)}-\bL'_{U_jh}\|_{\Ltwo} 
+ \frac{\bar h^2}{\sqrt{R}}\big(1.28\delta \|\bv_n\|_{\Ltwo}
+1.28\gamma h\delta \|\bg_n\|_{\Ltwo}
+1.35\gamma\delta\sqrt{p\delta}\big) \,.
\end{align}
Setting ${\color{blue}\square} = 1.28\delta \|\bv_n\|_{\Ltwo}
+1.28\gamma h\delta \|\bg_n\|_{\Ltwo}
+1.35\gamma\delta\sqrt{p\delta}$, the last display leads to
\begin{align}
\frac{1}{R}\sum_{i=1}^R  \|\bvartheta_n^{(q,i)} -  \bL'_{U_i h}\|_{\Ltwo}
    &\leqslant \frac{\bar h^2}{R}\sum_{j=1}^R\|\bvartheta_n^{(q-1,j)} 
    - \bL'_{U_jh}\|_{\Ltwo} + \frac{\bar h^2}{\sqrt{R}}\,
    {\color{blue}\square}.
\end{align}
We then find by induction that
\begin{align}
\frac{1}{R}\sum_{i=1}^R  \|\bvartheta_n^{(q,i)} -  \bL'_{U_i h}\|_{\Ltwo}
    &\leqslant {\bar h}^{2q}\frac{1}{R}\sum_{i=1}^R  \|\bvartheta_n -  \bL'_{U_i h}\|_{\Ltwo} + 
    \frac{\bar h^2}{\sqrt{R}} {\color{blue}\square} 
    \sum_{i=0}^{q-1} \bar h^{2i}\,.
\end{align}
Invoking display~\eqref{eq:prev2} again in conjunction with 
$\bL'_0 = \bvartheta_n$ gives
\begin{align}
\frac{1}{R}\sum_{i=1}^R  \|\bvartheta_n^{(q,i)} - \bL'_{U_i h}
    \|_{\Ltwo}
    &\leqslant \frac{\bar h^2}{\sqrt{R}}  {\color{blue}\square} 
    \sum_{i=0}^{q-1} \bar h^{2i}\\ 
    &\qquad + {\bar h}^{2q}e^{0.5{\bar h^2}} 
    \Big(h\|\bv_n\|_{\Ltwo} + 0.5\gamma h^2\|\bg_n\|_{\Ltwo}
   +\sqrt{(2/3)ph}\gamma h\Big)\,.
\end{align}
The desired result follows by setting $q=Q-1$.
}
\end{proof}

We also need the following auxiliary lemma.
\begin{lemma}\label{lem:dev}For every $h>0,$ it holds that
\begin{align}
\frac{1}{R}\sum_{i=1}^R \|\bvartheta_n^{(Q-1,i)}-\bvartheta_n\|_{\Ltwo}^2
& \leqslant
2\Big(1+2{\bar h^4}+\cdots+\big(2{\bar h^4}\big)^{Q-2}\Big) \big( h\|\bv_n\|_{\Ltwo} 
+ \gamma h^2 \|\bg_n\|_{\Ltwo}
+ \gamma h\sqrt{hp}\big)^2\,.
\end{align}
\end{lemma}

\begin{proof}[Proof of \Cref{lem:dev}]
By the definition of $\bvartheta_n^{(q,i)},$ it holds that
\begin{align}
\|\bvartheta_n^{(q,i)}-\bvartheta_n\|_{\Ltwo}
& \leqslant \Big\|\frac{1-e^{-\gamma U_i h}}{\gamma}\bv_n\Big\|_{\Ltwo}
+ \sum_{j=1}^R \Big\|\int_{(j-1)\delta}^{j\delta} (1-e^{-\gamma (U_ih-s)})\,\rmd s \bg_n^{(q-1,j)}\Big\|_{\Ltwo}
+\gamma h\sqrt{hp} \\
& \leqslant h\|\bv_n\|_{\Ltwo} + \sum_{j=1}^R\delta \gamma hM\big\| \bvartheta_n^{(q-1,j)}-\bvartheta_n\big\|_{\Ltwo}
+ R\delta\gamma h\|\bg_n\|_{\Ltwo}
+ \gamma h\sqrt{hp}\\
& = \frac{\bar h^2}{R}\sum_{i=1}^R \big\| \bvartheta_n^{(q-1,j)}-\bvartheta_n\big\|_{\Ltwo}
+  h\|\bv_n\|_{\Ltwo} 
+ \gamma h^2 \|\bg_n\|_{\Ltwo}
+ \gamma h\sqrt{hp}\,.
\end{align}
Squaring both sides then gives
\begin{align}
\|\bvartheta_n^{(q,i)}-\bvartheta_n\|_{\Ltwo}^2
\leqslant 2{\bar h^4} \frac{1}{R}\sum_{i=1}^R \|\bvartheta_n^{(q-1,i)}-\bvartheta_n\|_{\Ltwo}^2
+2 \Big( h\|\bv_n\|_{\Ltwo} 
+ \gamma h^2 \|\bg_n\|_{\Ltwo}
+ \gamma h\sqrt{hp}\Big)^2\,,
\end{align}
which implies
\begin{align}
\frac{1}{R}\sum_{i=1}^R \|\bvartheta_n^{(q,i)}-\bvartheta_n\|_{\Ltwo}^2
& \leqslant \frac{2{\bar h^4}}{R} \sum_{i=1}^R \|\bvartheta_n^{(q-1,i)}-\bvartheta_n\|_{\Ltwo}^2 + 2 \Big( h\|\bv_n\|_{\Ltwo} 
+ \gamma h^2 \|\bg_n\|_{\Ltwo}
+ \gamma h\sqrt{hp}\Big)^2\,.
\end{align}
By induction, we obtain
\begin{align}
\frac{1}{R}\sum_{i=1}^R \|\bvartheta_n^{(q,i)}-\bvartheta_n\|_{\Ltwo}^2
& \leqslant
2\Big(1+2{\bar h^4}+\cdots+\big(2{\bar h^4}\big)^{q-1}\Big) \big( h\|\bv_n\|_{\Ltwo} 
+ \gamma h^2 \|\bg_n\|_{\Ltwo}
+ \gamma h\sqrt{hp}\big)^2
\end{align}
as desired by setting $q=Q-1.$
\end{proof}

\begin{lemma}
\label{lem:dev1}
When $5M\leqslant \gamma, \gamma h\leqslant 0.1$, it holds that
\begin{align}
\frac{1}{R}\sum_{r=1}^R \|\bvartheta_n^{(Q-1,r)}-\bL'_{(r-1)\delta}\|_{\Ltwo} 
&\leqslant  
\bar h^{2Q-2} (1.003h \|\bv_n\|_{\Ltwo}
+ 0.11 h \|\bg_n\|_{\Ltwo}
+ 1.415\gamma h\sqrt{hp}) \\
&\quad + \delta(1.03 \|\bv_n\|_{\Ltwo}
 + 2.04\gamma h\|\bg_n\|_{\Ltwo}
+ 1.44\gamma \sqrt{hp})\,.
\end{align}
\end{lemma}
\begin{proof}
[Proof of \Cref{lem:dev1}]

By the triangle inequality, we have
\begin{align}
 \frac{1}{R}\sum_{r=1}^R \|\bvartheta_n^{(Q-1,r)} - 
 \bL'_{(r-1)\delta}\|_{\Ltwo} 
 &\leqslant 
  \frac{1}{R}\sum_{r=1}^R \|\bvartheta_n^{(Q-1,r)} - 
  \bL'_{U_rh}\|_{\Ltwo} + \frac{1}{R}\sum_{r=1}^R \|\bL'_{U_rh} 
  - \bL'_{(r-1)\delta}\|_{\Ltwo} \,.
\end{align}
Invoking~\Cref{lem:9} and inequality~\eqref{eq:prev2}, we find
{
\begin{align}
\frac{1}{R}\sum_{r=1}^R \|\bvartheta_n^{(Q-1,r)} -  \bL'_{(r-1)\delta} \|_{\Ltwo} 
 &\leqslant 
  \frac{\bar h^2}{\sqrt{R}}   \Big(\sum_{i=0}^{Q-2} \bar h^{2i}\Big)\big(1.28\delta \|\bv_n\|_{\Ltwo}
+1.28\gamma h\delta \|\bg_n\|_{\Ltwo}
+1.35\gamma\delta\sqrt{p\delta}\big)
   \\ 
    &\qquad + {\bar h}^{2(Q-2)}e^{0.5{\bar h^2}} 
    \Big(h\|\bv_n\|_{\Ltwo} + 0.5\gamma h^2\|\bg_n\|_{\Ltwo}
   +\sqrt{(2/3)ph}\gamma h\Big)\\
&\qquad + e^{\bar h^2/2}\Big(\delta\|\bv_n\|_{\Ltwo}
 + 0.5\gamma \delta^2\|\bg_n\|_{\Ltwo}
  + \gamma \delta\sqrt{2p\delta/3} \Big)
\end{align}
}
Noting that $\bar h^2\leqslant \gamma^2h^2/5\leqslant 0.002$ 
and $1+\bar h^2+\cdots+\bar h^{2Q-4}\leqslant \frac{1}{1- 
\bar h^2}\leqslant 1.003,$ it then follows that
\begin{align}
 \frac{1}{R}\sum_{r=1}^R 
 \|\bvartheta_n^{(Q-1,r)}-\bL'_{(r-1)\delta}\|_{\Ltwo} 
    &\leqslant \bar h^{2Q-2} \big(1.003h \|\bv_n\|_{\Ltwo}
        + 0.11 h \|\bg_n\|_{\Ltwo} + 1.415\gamma h\sqrt{hp} \big) \\
    &\quad + \delta \big(1.03 \|\bv_n\|_{\Ltwo} + 2.04\gamma h
    \|\bg_n\|_{\Ltwo} + 1.44\gamma \sqrt{hp} \big)
\end{align}
as desired.
\end{proof}

\allowdisplaybreaks

\subsubsection{Proof of \Cref{lem:2} (Discretization error)}

By the definition of $\bvartheta_{n+1},$ we 
have
\begin{align}
    \|\bar\bvartheta_{n+1}-\bL'_h\|_{\Ltwo}
    & = \Big\|\mathbb E\Big[\sum_{i=1}^R \delta
    \big(1-e^{-\gamma h(1-U_{i})}\big)\nabla f 
    (\bvartheta_n^{(Q-1,i)})\Big]- 
    \int_0^h\big(1-e^{-\gamma(h-s)}\big)\nabla
    f(\bL'_s)\,\rmd s\Big\|_{\Ltwo} \\
    & = \Big\| \sum_{i=1}^R \delta \E\Big[
    \big(1-e^{-\gamma h(1-U_{i})}\big)\nabla f 
    (\bvartheta_n^{(Q-1,i)})\Big]
    -\sum_{i=1}^R\int_{(i-1)\delta}^{i\delta}  \big(1-e^{-\gamma (h-s)}\big)\grad f(\bL'_s)\,\rmd s \Big\|_{\Ltwo} \\
    & = \Big\| \sum_{i=1}^R \delta \E\Big[
    \big(1-e^{-\gamma h(1-U_{i})}\big)\nabla f 
    (\bvartheta_n^{(Q-1,i)})\Big]
    -\sum_{i=1}^R\delta\E\Big[  \big(1-e^{-\gamma h(1-U_{i})}\big)\grad f(\bL'_{U_ih}) \Big]\Big\|_{\Ltwo} \\
    &\leqslant \sum_{i=1}^R\delta   M\gamma h \|
    \bvartheta_n^{(Q-1,i)} - \bL'_{U_i h}\|_{\Ltwo} ,
\end{align}
where the last inequality follows from the 
smoothness of function $f$ and the fact that 
$1-e^{-\gamma(h-U h)}\leqslant \gamma (1 - U)h$. 
By \Cref{lem:9}, we then obtain
{
\begin{align}
\|\bar\bvartheta_{n+1}-\bL'_h\|_{\Ltwo} 
& \leqslant 
\frac{\bar h^4}{\sqrt{R}}   \Big(\sum_{i=0}^{Q-2} \bar h^{2i}\Big)\big(1.28\delta \|\bv_n\|_{\Ltwo}
+1.28\gamma h\delta \|\bg_n\|_{\Ltwo}
+1.35\gamma\delta\sqrt{p\delta}\big)
   \\ 
    &\qquad + {\bar h}^{2Q}e^{0.5{\bar h^2}} 
    \Big(h\|\bv_n\|_{\Ltwo} + 0.5\gamma h^2\|\bg_n\|_{\Ltwo}
   +\sqrt{(2/3)ph}\gamma h\Big)\,.
\end{align}
}
Multiplying both sides by $\gamma$ completes the proof of the first claim.

Using the definition of $\bvartheta_{n+1}$, and 
the fact that the mean minimizes the squared 
integrated error, we get
\begin{align}
\|\bvartheta_{n+1}-\bar\bvartheta_{n+1}\|_{\Ltwo}
&\leqslant   \bigg\|\sum_{r=1}^R \delta (1-e^{\gamma h(1-U_r)})\grad f(\bvartheta_n^{(Q-1,r)})-\sum_{r=1}^R \E[\delta (1-e^{\gamma h(1-U_r)})]\grad f(\bL'_{(r-1)\delta})\bigg\|_{\Ltwo}\\
&\leqslant \bigg\|\sum_{r=1}^R \delta (1-e^{-\gamma h(1-U_r)})\Big[\grad f(\bvartheta_n^{(Q-1,r)})-\grad f(\bL'_{(r-1)\delta})\Big] \bigg\|_{\Ltwo}\\
    &\quad +\bigg\|\sum_{r=1}^R \delta\Big(e^{-\gamma h(1-U_r)} - 
    \mathbb E [e^{-\gamma h(1-U_r)}]\Big)\grad f(\bL'_{(r-1)\delta})
    \bigg\|_{\Ltwo}\\
    &\leqslant \delta \gamma hM\sum_{r=1}^R\|\bvartheta_n^{(Q-1,r)}-\bL'_{(r-1)\delta}\|_{\Ltwo}  + \sum_{r=1}^R \gamma\delta^2 \|\grad f(\bL'_{(r-1)\delta})\|_{\Ltwo}\\
    & =\bar h^2\frac{1}{R}\sum_{r=1}^R\|\bvartheta_n^{(Q-1,r)}-\bL'_{(r-1)\delta}\|_{\Ltwo} 
+ \sum_{r=1}^R \gamma \delta^2  \|\grad f(\bL'_{(r-1)\delta})\|_{\Ltwo}\,.
\label{eq:term2}
\end{align}
The second to last inequality follows from the fact that $|e^{-x}-e^{-y}|\leqslant |x-y|,\forall x,y>0.$ 
Moreover, by display~\eqref{eq:prev2}, it holds that
\begin{align}
\|\grad f(\bL'_{(r-1)\delta})\|_{\Ltwo}
&\leqslant \|\grad f(\bL'_{(r-1)\delta})-\grad f(\bL'_0)\|_{\Ltwo}+\|\grad f(\bL'_0)\|_{\Ltwo}\\
&\leqslant M \|\bL'_{(r-1)\delta}-\bL'_0\|_{\Ltwo}+\|\grad f(\bL'_0)\|_{\Ltwo}\\
&\leqslant e^{0.5\bar h^2}Mh\Big(\|\bv_n\|_{\Ltwo}+0.5\gamma h\|\bg_n\|_{\Ltwo}+\sqrt{2/3}\sqrt{hp}\gamma \Big)+\|\bg_n\|_{\Ltwo}\,.
\end{align}
Combining this with the previous display~\eqref{eq:term2} then gives
\begin{align}
    \|\bvartheta_{n+1}-\bar\bvartheta_{n+1}\|_{\Ltwo}
    &\leqslant \bar h^2 \frac{1}{R}\sum_{r=1}^R 
    \|\bvartheta_n^{(Q-1,r)}-\bL'_{(r-1)\delta}\|_{\Ltwo} \\
    &\quad + R\gamma \delta^2 \Big\{e^{0.5\bar h^2}Mh\Big(\|\bv_n\|_{\Ltwo}+0.5\gamma h\|\bg_n\|_{\Ltwo} 
    + \sqrt{2/3}\sqrt{hp}\gamma \Big)+\|\bg_n\|_{\Ltwo} \Big\}.
\end{align}
Recall that $M\leqslant \gamma/5$ and $\gamma h\leqslant 0.1,$ it then holds that
\begin{align}
\|\bvartheta_{n+1}-\bar\bvartheta_{n+1}\|_{\Ltwo}
&\leqslant \bar h^2 \frac{1}{R}\sum_{r=1}^R \|\bvartheta_n^{(Q-1,r)}-\bL'_{(r-1)\delta}\|_{\Ltwo} \\
&\quad + \gamma h\delta \Big(0.021\|\bv_n\|_{\Ltwo}+1.01\|\bg_n\|_{\Ltwo}+ 0.02\gamma \sqrt{hp} \Big)\,.
\end{align}
By \Cref{lem:dev1}, we obtain
\begin{align}
\|\bvartheta_{n+1}-\bar\bvartheta_{n+1}\|_{\Ltwo}
&\leqslant \bar h^{2Q} (1.003h \|\bv_n\|_{\Ltwo}
+ 0.11 h \|\bg_n\|_{\Ltwo}
+ 1.415\gamma h\sqrt{hp}) \\
&\qquad + M \gamma h^2\delta(1.03 \|\bv_n\|_{\Ltwo}
 + 2.04\gamma h\|\bg_n\|_{\Ltwo}
+ 1.44\gamma \sqrt{hp})\\
&\qquad + \gamma h\delta\Big(0.021\|\bv_n\|_{\Ltwo}+1.01\|\bg_n\|_{\Ltwo}+ 0.02\gamma \sqrt{hp} \Big)\\
&\leqslant 
\bar h^{2Q} (1.003h \|\bv_n\|_{\Ltwo}
+ 0.11 h \|\bg_n\|_{\Ltwo}
+ 1.415\gamma h\sqrt{hp}) \\
&\qquad + \gamma h\delta (0.22\|\bv_n\|_{\Ltwo}
+ 1.02 \|\bg_n\|_{\Ltwo}
+ 0.05\gamma \sqrt{hp}  )\\
&\leqslant (\gamma h)^{Q-1} (0.01 h \|\bv_n\|_{\Ltwo}
+ 0.01 h \|\bg_n\|_{\Ltwo}
+ 0.01 \gamma h\sqrt{hp})\\
&\qquad + \gamma h\delta (0.22\|\bv_n\|_{\Ltwo}
+ 1.02 \|\bg_n\|_{\Ltwo}
+ 0.05\gamma \sqrt{hp}  )
\,.
\end{align}
The last step follows from $\gamma h\leqslant 0.1$ and $5M\leqslant \gamma$.
Multiplying both sides by $\gamma$ completes the proof of the second claim in this lemma.

Using the same trick as used in the treatment of the first claim, by the definition of $\bv_{n+1},$ we have
\begin{align}
\|\bar\bv_{n+1}-\bV'_{h}\|_{\Ltwo}
&=\gamma \bigg\|\sum_{i=1}^R\delta \E\Big[e^{-\gamma h(1-U_i)}\grad f\big(\bvartheta_n^{(Q-1,i)}\big)\Big] -\int_0^h e^{-\gamma(h-s)}\grad f(\bL'_s)\,\rmd s  \bigg\|_{\Ltwo}\\
&\leqslant \gamma \delta \sum_{i=1}^R \bigg\| e^{-\gamma h(1-U_i)}\Big[\grad f\big(\bvartheta_n^{(Q-1,i)}\big)-\grad f(\bL'_{U_ih})\Big] \bigg\|_{\Ltwo}\\
&\leqslant  \frac{\gamma Mh}{R} \sum_{i=1}^R \|\bvartheta_n^{(Q-1,i)}-\bL'_{U_ih}\|_{\Ltwo} \,.
\end{align}
By \Cref{lem:9}, we obtain the third claim of the lemma.

The proof of the fourth claim follows readily by employing the same technique as the proof for the second claim.

\subsubsection{Proof of \Cref{lem:8}}
From \eqref{eq:thetaU}, \eqref{eq:thetan+1}  and
\eqref{eq:vn+1}, using the notation $\varrho_i = 
e^{-\gamma h(1-U_i)}$ it follows that
\begin{align}
\|\bv_{n+1}\|_{\Ltwo}^2 
& \leqslant (1-\bar\alpha_2\eta)^2\|\bv_n\|_{\Ltwo}^2 
    +\gamma^2\bigg\| \sum_{i=1}^R \delta \varrho_i
    \bg_n^{(Q-1,i)}\bigg\|^2_{\Ltwo} + 2\gamma \eta p\\
    &\qquad - 2(1 - \bar\alpha_2 \eta ) \mathbb E\bigg[
    \bv_n^\top\gamma \sum_{i=1}^R\delta \varrho_i
    \bg_n^{(Q-1,i)}\bigg] + \sqrt{2\gamma \eta p}\bigg\| 
    \sum_{i=1}^R \gamma\delta \varrho_i\big( 
    \bg_n^{(Q-1,i)}-\bg_n\big)\bigg\|_{\Ltwo} \\
&\leqslant  (1-\bar\alpha_2\eta  
        )^2\|\bv_n\|_{\Ltwo}^2 
        +2\delta\gamma^2\bigg\{\bigg\| \sum_{i=1}^R e^{-\gamma 
        h(1-U_i)}\big(\bg_n^{(Q-1,i)}-\bg_n\big)\bigg\|^2_{\Ltwo} 
        + \bigg\| \sum_{i=1}^R \varrho_i\bg_n
        \bigg\|^2_{\Ltwo}\bigg\}\\
        &\qquad+2\gamma \eta p
       -2(1-
       \bar\alpha_2 \eta ) \mathbb E\bigg[\bv_n^\top\gamma \sum_{i=1}^R\delta \varrho_i\big(\bg_n^{(Q-1,i)}-\bg_n\big)\bigg] \\
       &\qquad  -2(1- \bar\alpha_2 \eta ) \mathbb E\bigg[\bv_n^\top 
       \gamma \sum_{i=1}^R\delta \varrho_i\bg_n\bigg]  + \gamma \eta p+\frac{\gamma^2}{2}\bigg\|\sum_{i=1}^R\delta \varrho_i\big(\bg_n^{(Q-1,i)}-\bg_n\big)\bigg\|_{\Ltwo} ^2\\
&\leqslant (1-\bar\alpha_2\eta  
        )^2\|\bv_n\|_{\Ltwo}^2 
        +3\gamma \eta p
        +2.5 \gamma^2 \bigg\|\sum_{i=1}^R\delta \varrho_i\big(\bg_n^{(Q-1,i)}-\bg_n\big)\bigg\|_{\Ltwo} ^2 \\
        &\qquad + 2\gamma^2h^2\big\| \bg_n\big\|^2_{\Ltwo}
        -2(1-
       \bar\alpha_2 \eta )\gamma \delta \mathbb E\Big[\bv_n^\top \bg_n\Big]\E\bigg[ \sum_{i=1}^R \varrho_i\bigg] 
        \\
        &\qquad -2(1-
       \bar\alpha_2 \eta ) \mathbb E\bigg[\bv_n^\top\gamma \sum_{i=1}^R\delta \varrho_i\big(\bg_n^{(Q-1,i)}-\bg_n\big)\bigg] \\
&\leqslant  (1-\bar\alpha_2\eta  
        )^2\|\bv_n\|_{\Ltwo}^2 
        +3\gamma \eta p
         + 2\gamma^2h^2 \big\| \bg_n\big\|^2_{\Ltwo}
         +2.5M^2\gamma^2 \frac{h^2}{R} \sum_{i=1}^R  \big\| \bvartheta_n^{(Q-1,i)} -\bvartheta_n\big\|^2_{\Ltwo}
       \\
        &\qquad 
       +2M\big(\sqrt{\gamma }h\|\bv_n\|_{\Ltwo} \big)\Big(\sqrt{\gamma} \frac{1}{R}\sum_{i=1}^R \big\|\bvartheta_n^{(Q-1,i)}-\bvartheta_n\big\|_{\Ltwo}\Big)
       -2 \bar\alpha_2 \eta(1-
       \bar\alpha_2 \eta ) \mathbb E[\bv_n^\top\bg_n] \\
&\leqslant (1-\bar\alpha_2\eta  
        )^2\|\bv_n\|_{\Ltwo}^2 
        +3\eta\gamma p
         + 2\eta^2 \big\| \bg_n\big\|^2_{\Ltwo}
         +{\bar h^2}\|\bv_n\|_{\Ltwo}^2\\
       & \qquad -2 \bar\alpha_2\eta(1-
       \bar\alpha_2 \eta ) \mathbb E[\bv_n^\top\bg_n] 
       + 2.05M\gamma \big(h\|\bv_n\|_{\Ltwo}+\gamma h^2\|\bg_n\|+\gamma h\sqrt{hp}\big)^2\,.
    \end{align}
The last but one inequality makes use of the smoothness of function $f,$ Cauchy–Schwarz inequality, and the fact that $\sum_{i=1}^R\E[e^{-\gamma h(1-U_i)}]=R\bar\alpha_2.$
The last inequality follows from \Cref{lem:dev}. 
Since for $\eta\leqslant 0.1,$ we have~$\bar\alpha_2\geqslant 0.95$, this implies
\begin{align}
(1 - \bar\alpha_2\eta )^2 + {\bar h^2}+6.15M \gamma h^2 
 \leqslant (1-0.95\eta)^2 +0.143\eta
 \leqslant 1-1.80975\eta+0.143\eta
 \leqslant 1- 1.66\eta\,.
\end{align}
Therefore, we obtain
\begin{align}
\|\bv_{n+1}\|_{\Ltwo}^2 & \leqslant {(1 - 1.66 \eta ) \| 
\bv_n\|_{\Ltwo}^2 - 2\bar\alpha_2\eta  (1 - 
\eta \bar\alpha_2) \mathbb E [\bv_n^\top \bg_{n}] + 2.02\eta^2 \|\bg_n\|_{\Ltwo}^2 +
3.02\eta\gamma p}
\label{eq:43}
\end{align}
as desired.

The next step is to get an upper bound on 
    $\mathbb E[\bv_{n+1}^\top\bg_{n+1}] - \mathbb
    E[\bv_{n}^\top\bg_n]$ in order to prove
    \eqref{eq:help2}. To this end, we first
    note that
 \begin{align}
        \mathbb E[\bv_{n+1}^\top\bg_{n+1}]
&\leqslant \mathbb E[\bv_{n}^\top\bg_n]  + 
           \mathbb E[\bv_{n+1}^\top(\bg_{n+1}-\bg_n)] + \mathbb E[(\bv_{n+1} -\bv_n)^\top\bg_n]\\
&\leqslant \mathbb E[\bv_{n}^\top\bg_n] +  
        M\|\bv_{n+1}\|_{\Ltwo}\|\bvartheta_{n+1}-\bvartheta_n\|_{\Ltwo} -\bar\alpha_2\eta \mathbb E[\bv_n^\top\bg_n] \\
        &\qquad - \gamma\E\Big[\bg_n^\top\sum_{i=1}^R\delta e^{-\gamma h(1-U_i)}\grad f\big(\bvartheta_n^{(Q-1,i)}\big)\Big]\\
&\leqslant (1- \bar\alpha_2\eta )\mathbb E[\bv_{n}^\top\bg_n] 
         + M\|\bv_{n+1}\|_{\Ltwo}\|\bvartheta_{n+1}-\bvartheta_n\|_{\Ltwo} \\
         &\qquad
         -\gamma \E\Big[\bg_n^\top \sum_{i=1}^R\delta  e^{-\gamma h(1-U_i)}\bg_n + \bg_n^\top \sum_{i=1}^R\delta e^{-\gamma h(1-U_i)}\big(\grad f(\bvartheta_n^{(Q-1,i)})-\bg_n\big)\Big] \\
& \leqslant (1- \bar\alpha_2\eta )\mathbb E[\bv_{n}^\top\bg_n] 
        +  M\|\bv_{n+1}\|_{\Ltwo}\|\bvartheta_{n+1}-\bvartheta_n\|_{\Ltwo} 
        -\bar\alpha_2\eta \|\bg_n\|_{\Ltwo}^2  \\
        &\qquad + M\gamma h \|\bg_{n}\|_{\Ltwo}\frac{1}{R}\sum_{i=1}^R \|\bvartheta_n^{(Q-1,i)}-\bvartheta_n\|_{\Ltwo} \\
& \leqslant (1- \bar\alpha_2\eta )\mathbb E[\bv_{n}^\top\bg_n] 
        +  M\|\bv_{n+1}\|_{\Ltwo}\|\bvartheta_{n+1}-\bvartheta_n\|_{\Ltwo} 
        -\bar\alpha_2\eta \|\bg_n\|_{\Ltwo}^2  \\
        &\qquad + 1.4143 M\gamma h \|\bg_{n}\|_{\Ltwo} \big(h\|\bv_n\|_{\Ltwo}+\gamma h^2\|\bg_n\|_{\Ltwo}+\eta\sqrt{hp}\big) \\
& \leqslant (1- \bar\alpha_2\eta )\mathbb E[\bv_{n}^\top\bg_n] 
        +  M\|\bv_{n+1}\|_{\Ltwo}\|\bvartheta_{n+1}-\bvartheta_n\|_{\Ltwo} 
        -\bar\alpha_2\eta \|\bg_n\|_{\Ltwo}^2  \\
        &\qquad + 1.001 \gamma^2h^2 Mh  \|\bg_{n}\|^2_{\Ltwo} 
        + 0.5Mh\big(\|\bv_n\|_{\Ltwo}+\gamma h\|\bg_n\|_{\Ltwo}+\gamma\sqrt{hp}\big)^2    
\label{eq:inprod}
         \,.
    \end{align}
The last but one display follows from \Cref{lem:dev}.
We now aim to derive the bounds for $ \|\bv_{n+1}\|_{\Ltwo}$ and $\|\bvartheta_{n+1}-\bvartheta_n\|_{\Ltwo}$.
We first note that
\begin{align}
  \|\bv_{n+1}\|_{\Ltwo}
&\leqslant \|\bv_n\|_{\Ltwo} 
        + \gamma\sum_{i=1}^R\delta \big\|\grad f(\bvartheta_n^{(Q-1,i)})\big\|_{\Ltwo}
        +\sqrt{2\eta\gamma p}\\  
&\leqslant  \|\bv_n\|_{\Ltwo} 
        + \gamma\sum_{i=1}^R\delta \big\|\grad f(\bvartheta_n^{(Q-1,i)})-\bg_n\big\|_{\Ltwo}
        +\gamma h\|\bg_n\|_{\Ltwo}
        +\sqrt{2\eta\gamma p}\\ 
&\leqslant  \|\bv_n\|_{\Ltwo} 
        + \gamma M h\frac{1}{R}\sum_{i=1}^R\big\|\bvartheta_n^{(Q-1,i)}- \bvartheta_n\big\|_{\Ltwo}
        +\gamma h\|\bg_n\|_{\Ltwo}
        +\sqrt{2\eta\gamma p}\\
&\leqslant \|\bv_n\|_{\Ltwo} 
        + 1.4143\gamma M h\big(h\|\bv_n\|_{\Ltwo}+\gamma h^2\|\bg_n\|_{\Ltwo}+\eta\sqrt{hp}\big)
        +\gamma h\|\bg_n\|_{\Ltwo}
        +\sqrt{2\eta\gamma p}\,.
\end{align}
The last inequality follows from \Cref{lem:dev}. 
Note that $M\geqslant 5\gamma$ and $\eta\leqslant 0.1,$ rearranging the display gives
    \begin{align}
        \|\bv_{n+1}\|_{\Ltwo}
        &\leqslant 1.003\big(\|\bv_n\|_{\Ltwo} +  
        \eta\|\bg_n\|_{\Ltwo} 
        +\sqrt{2\eta\gamma p}\big).\label{eq:46}
    \end{align}
Employing \Cref{lem:dev}, we also infer that
\begin{align}
\|\bvartheta_{n+1}-\bvartheta_n\|_{\Ltwo}
&\leqslant h\|\bv_n\|_{\Ltwo} + \sum_{i=1}^R\delta \gamma h \big\|\grad f(\bvartheta_n^{(Q-1,i)})\big\|_{\Ltwo} +\eta\sqrt{hp}\\
&\leqslant  h\|\bv_n\|_{\Ltwo} 
+ \frac{\bar h^2}{R}\sum_{i=1}^R \big\|\bvartheta_n^{(Q-1,i)}-\bvartheta_n\big\|_{\Ltwo} 
+ \gamma h^2 \|\bg_n\|_{\Ltwo} 
+\eta\sqrt{hp}\\
&\leqslant  h\|\bv_n\|_{\Ltwo} 
+ 1.4143 {\bar h^2} \big(h\|\bv_n\|_{\Ltwo}+\gamma h^2\|\bg_n\|_{\Ltwo}+\eta\sqrt{hp}\big) 
+ \gamma h^2 \|\bg_n\|_{\Ltwo} 
+\eta\sqrt{hp}\\
&\leqslant 1.003 \big(h \|\bv_n\|_{\Ltwo} + \gamma h^2\|\bg_n\|_{\Ltwo} +\eta\sqrt{hp} \big)\,.
\end{align}
Thus, we obtain
\begin{align}
\|\bv_{n+1}\|_{\Ltwo}\|\bvartheta_{n+1}-\bvartheta_n\|_{\Ltwo}
\leqslant 3.02h\big(\|\bv_n\|_{\Ltwo}^2+  
        \eta^2\|\bg_n\|_{\Ltwo}^2 + 
        2\eta\gamma  p\big)\,.
\end{align}
Plugging this back into the display~\eqref{eq:inprod} provides us with
\begin{align}
  \mathbb E[\bv_{n+1}^\top\bg_{n+1}]
& \leqslant (1- \bar\alpha_2\eta )\mathbb E[\bv_{n}^\top\bg_n] 
        +  3.02Mh\big(\|\bv_n\|_{\Ltwo}^2+  
        \eta^2\|\bg_n\|_{\Ltwo}^2 + 
        2\eta\gamma  p\big)
        -\bar\alpha_2\eta \|\bg_n\|_{\Ltwo}^2  \\
        &\qquad + 1.001 \gamma^2h^2 Mh  \|\bg_{n}\|^2_{\Ltwo} 
        + 0.5Mh\big(\|\bv_n\|_{\Ltwo}+\gamma h\|\bg_n\|_{\Ltwo}+\gamma\sqrt{hp}\big)^2  \\
& \leqslant (1- \bar\alpha_2\eta )\mathbb E[\bv_{n}^\top\bg_n] 
- 0.938\eta
        \|\bg_n\|_{\Ltwo}^2  + 0.91  \eta\|\bv_n\|_{\Ltwo
        }^2 +  1.51 \gamma \eta^2 p.\label{eq:47}
\end{align}
as desired. 

Lastly, using the Lipschitz property of $\nabla f$, 
\Cref{lem:dev} and the notation ${\color{blue}\triangle} = 
0.502M\gamma\big(h\|\bv_n\|_{\Ltwo}+\gamma h^2\|\bg_n\|_{\Ltwo}+\gamma h\sqrt{hp}\big)^2$, we have
 \begin{align}
\gamma\mathbb E[f_{n+1} - f_n] 
& \leqslant \gamma\mathbb E[ \bg_n^\top(\bvartheta_{n+1} 
        - \bvartheta_n)] + (M \gamma/2)\| (\bvartheta_{n+1} -
        \bvartheta_n) \|^2_{\Ltwo}\\
& \leqslant \E\bigg[ \bg_n^\top \Big(\bar\alpha_2 \gamma h\bv_n -\gamma \sum_{i=1}^R\gamma h \delta \frac{1-e^{-\gamma h(1-U_i)}}{\gamma h}\grad f(\bvartheta_n^{(Q-1,i)}) \Big)\bigg]  +  {\color{blue}\triangle} \\
& = \bar\alpha_2\gamma h\E[\bv_n^\top \bg_n] + {\color{blue}\triangle}\\
    &\qquad -\gamma^2\frac{h^2}{R}\sum_{i=1}^R \E\bigg[
    \frac{1-e^{-\gamma h(1-U_i)}}{\gamma h}\bigg\{\|\bg_n\|_{\Ltwo}^2 + 
    \bg_n^\top\big(\grad f(\bvartheta_n^{(Q-1,i)})-\bg_n\big)\bigg\}\bigg]   \,.
\end{align}
Note that 
\begin{align}
  \frac{1}{R}\sum_{i=1}^R \E\Big[\frac{1-e^{-\gamma h(1-U_i)}}{\gamma h}\Big] 
= \frac{\gamma h-1+e^{\gamma h}}{(\gamma h)^2}\in [0.484,0.5),\qquad\text{ and }\qquad
 \frac{1-e^{-\gamma h(1-U_i)}}{\gamma h} \leqslant 1\,.
\end{align}
It then follows that     
\begin{align}
\gamma\mathbb E[f_{n+1} - f_n] 
& \leqslant \bar\alpha_2\gamma h\E[\bv_n^\top \bg_n] -0.484\eta^2 \|\bg_n\|_{\Ltwo}^2 +0.5 \gamma^4h^4\|\bg_n\|_{\Ltwo}^2
+ \frac{M^2}{2R}\sum_{i=1}^R\|\bvartheta_n^{(Q-1,i)}-\bvartheta_n
\|_{\Ltwo}^2 + {\color{blue}\triangle} \\
&\leqslant  \bar\alpha_2\gamma h\E[\bv_n^\top \bg_n] -0.484\eta^2 \|\bg_n\|_{\Ltwo}^2 + 0.5 \gamma^4h^4\|\bg_n\|_{\Ltwo}^2 +
1.001M^2{\color{blue}\triangle}    + {\color{blue}\triangle} \\
&\leqslant  \bar\alpha_2\gamma h\E[\bv_n^\top \bg_n] -0.474\eta^2 \|\bg_n\|_{\Ltwo}^2 +0.33\eta^2 \|\bv_n\|_{\Ltwo}^2+0.324\eta^3\gamma p\,.
\end{align}
This completes the proof of the lemma.

\end{document}